\documentclass[12pt,a4paper, oneside,reqno]{article}
\usepackage[utf8]{inputenc}
\usepackage[english]{babel}
\usepackage[left=1in,right=1in,top=1in,bottom=1in]{geometry}
\usepackage{amsfonts,amssymb,amsmath,amsthm}

\usepackage[symbol]{footmisc} 
\usepackage{longtable}
\usepackage{tasks}
\usepackage{enumerate}
\usepackage{cite}
\usepackage{hyperref}
\usepackage{verbatim}
\usepackage{tkz-graph}
\usepackage{fancyhdr} 
\theoremstyle{plain}
\newtheorem{theorem}{Theorem}
\newtheorem{lemma}[theorem]{Lemma}
\newtheorem{proposition}[theorem]{Proposition}
\newtheorem{remark}[theorem]{Remark}

\newtheorem{definition}[theorem]{Definition}
\newtheorem{corollary}[theorem]{Corollary}
\usetikzlibrary{arrows}
\newtheorem*{theoremM}{Main Theorem}

 \usepackage{fouriernc} 

\newcommand{\cS}{{S}}
\newcommand{\cB}{\mathcal{B}}
\newcommand{\Der}{\mathfrak{Der}}
\newcommand{\QDer}{{\rm Q}\mathfrak{Der}}
\newcommand{\GDer}{{\rm G}\mathfrak{Der}}
\newcommand{\Inn}{\operatorname{Inn}}
\newcommand{\ad}{\operatorname{ad}}
\newcommand{\id}{\mathrm{id}}
\newcommand{\hg}{\operatorname{hg}}
\newcommand{\Ann}{\operatorname{Ann}}
\newcommand{\Th}{\Theta}

\begin{document}

\noindent{\large
Derivations and quasi-derivations of Block Lie superalgebras $S(q)$}\footnote{The authors thank the SRMC (Sino-Russian Mathematics Center at Peking University, Beijing, China) for its hospitality and excellent working conditions, where some parts of this work were done. 
The work is supported by FCT 2023.08031.CEECIND and UID/00212/2025.}  
\footnote{The present version of the paper was checked in the proof assistant Lean~4 (with the Mathlib library) with the help of Aristotle (\url{https://aristotle.harmonic.fun/}). }

\bigskip

\bigskip

\begin{center}	
	{\bf
		Ivan Kaygorodov\footnote{CMA-UBI, University of  Beira Interior,  Portugal;  \    kaygorodov.ivan@gmail.com},
   Abror Khudoyberdiyev\footnote{V.I.Romanovskiy Institute of Mathematics, Uzbekistan Academy of Sciences; National University of Uzbekistan, Uzbekistan; \ 
khabror@mail.ru} \&
Zarina Shermatova\footnote{V.I.Romanovskiy Institute of Mathematics, Uzbekistan Academy of Sciences; Kimyo International University in Tashkent, Uzbekistan;  \ 
zarina\_shermatova91@mail.ru}}   
\end{center}

\bigskip

\bigskip

\noindent {\bf Abstract.}
{\it
We compute the homogeneous derivations and the homogeneous quasi-derivations of the Block Lie superalgebras $S(q)$,
whose even part is the Block Lie algebra $\mathcal{B}(q)$.
For every $q\in\mathbb C$, the space $\mathfrak{Der}_\theta(S(q))$ is the direct sum of the inner derivations and an explicitly given
finite-dimensional space of outer derivations; in particular, for $q\neq0$ a new odd outer derivation appears if and only if $q\in2\mathbb Z^*$.
For $q\in\mathbb C^*$ we prove that $\mathfrak{QDer}(S(q))=\mathfrak{Der}(S(q))\oplus\langle \id\rangle$ if $q\notin\mathbb Z$, and that
for $q\in\mathbb Z^{*}$ the superalgebra $S(q)$ has an infinite-dimensional space of nontrivial quasi-derivations,
namely all graded linear maps with values in the center.
For $q=0$ we also describe all homogeneous quasi-derivations: modulo maps with values in the center, they are spanned by
$\id$, three further maps and infinite families of maps shifting the $\mathbb Z^2$-grading; in particular,
$\mathfrak{QDer}(S(0))/\mathfrak{Der}(S(0))$ is infinite-dimensional.
As corollaries, we obtain the complete description of homogeneous $\delta$-derivations of $S(q)$, $q\neq0$, and prove that,
in contrast with the Block Lie algebras $\mathcal{B}(q)$, the superalgebras $S(q)$, $q\neq0$, do not admit nontrivial transposed
$\delta$-Poisson superalgebra structures for any $\delta\notin\{0,1\}$.}

\bigskip


\bigskip


\tableofcontents

\section*{Introduction}

Generalized derivations and quasi-derivations of Lie algebras were introduced by Leger and Luks \cite{LL}.
They generalize both derivations and Filippov's $\delta$-derivations \cite{f98,f95,zz}.
The description of $\delta$-derivations (and, more generally, of quasi-derivations) gives a constructive way to classify
transposed $\delta$-Poisson structures on a given Lie algebra \cite{delta,deltu,FKL}; for $\delta=\frac12$ these are the transposed Poisson structures.
About recent studies of transposed Poisson structures, see \cite{kkz1,kkz4,aae23,LS26,kms,S24} or the survey \cite{k23}. 
Recently, quasi-derivations of the Witt algebra and of the algebras $\mathcal W(a,b)$ \cite{kkz2},
of the deformed Heisenberg--Virasoro algebra \cite{LD26}, and of the $n$-th Schr\"odinger superalgebra \cite{WW26}
were described, and the quasi-derivations of the Block Lie algebras $\cB(q)$ were computed in \cite{kkz3}.
In the last case, $\cB(q)$ has a unique (up to isomorphism) nontrivial transposed $\delta$-Poisson structure when
$q\in\mathbb Z$; it comes from the $\delta$-derivation $D_5\colon L_{0,-2q}\mapsto L_{0,-q}$.

\medskip

Block type algebras were introduced by Block \cite{block58} and studied in many papers
\cite{dz96,oz99,xu99,xyz12,cgz14,gwl21,suxia20}.
Derivations of generalized Block algebras were described by \DJ okovi\'c and Zhao \cite{dz96}, and biderivations of the Block Lie algebras
by Liu, Guo and Zhao \cite{lgz18}.
Xia \cite{xia16} introduced two classes of Lie superalgebras of Block type and studied, in particular, their derivations;
for instance, it was shown there that every derivation of the Lie superalgebra of Block type considered in \cite{xia16} is the sum of an inner derivation
and a multiple of the outer derivation $d$ of degree zero given by $d\colon L_{m,i}\mapsto iL_{m,i}$, $G_{\alpha,j}\mapsto jG_{\alpha,j}$ (in the notation of \cite{xia16}).
In the present paper we consider the Lie superalgebra $\cS(q)$ whose even part is the Block algebra $\cB(q)$ of Xia, You and Zhou \cite{xyz12},
with arbitrary $q\in\mathbb C$ and basis indexed by the whole lattice $\mathbb Z\times\mathbb Z$.
Transposed Poisson structures on Block Lie algebras and superalgebras were studied in \cite{kk22,HGS25}.

\medskip

Our main results are the following.
\begin{itemize}
\item For every $q\in\mathbb C$ we describe all homogeneous derivations of $\cS(q)$ (Theorem \ref{thm:main}).
The outer part $\Der(\cS(q))/\Inn(\cS(q))$ has dimension $1$ if $q\notin\mathbb Z$, $3$ if $q$ is an odd integer,
$4$ if $q\in2\mathbb Z^*$ and $10$ if $q=0$. In particular, the derivation $D_5$ of $\cB(q)$ does not extend to $\cS(q)$ for $q\in\mathbb Z^*$, while
for $q\in2\mathbb Z^*$ a new odd derivation $D_6\colon G_{0,-\frac{3q}{2}}\mapsto L_{0,-q}$ appears.
\item For $q\neq0$ we describe all homogeneous quasi-derivations of $\cS(q)$ (Theorems \ref{thm:qnotZ} and \ref{thm:qZ}):
every quasi-derivation is the sum of a derivation, a multiple of the identity map and a graded map with values in the center.
\item For $q=0$ we describe all homogeneous quasi-derivations of $\cS(0)$ (Theorem \ref{thm:q0}):
\begin{center}
$\QDer_{\bar0}(\cS(0))=\mathcal V_{\bar0}\oplus\langle\id,T\rangle\oplus\Ann^Q_{\bar0}(\cS(0))$
\end{center} and
\begin{center}$\QDer_{\bar1}(\cS(0))=\mathcal V_{\bar1}\oplus\langle\Th,\Th'\rangle\oplus\Ann^Q_{\bar1}(\cS(0))$, where $\mathcal V_{\bar0}$
\end{center} is spanned by the maps
$V_{\beta,\kappa,\lambda}\colon X_{m,i}\mapsto(\kappa m+\lambda i)X_{m+\beta_1,i+\beta_2}$, $\mathcal V_{\bar1}$ is spanned by the maps
$V^o_{\beta,\kappa,\lambda}\colon L_{m,i}\mapsto(\kappa m+\lambda i)G_{m+\beta_1,i+\beta_2}$, $G_{m,i}\mapsto0$, and $\Ann^Q_\theta(\cS(0))$ consists of the maps with values in the center.
In particular, $\QDer(\cS(0))/\Der(\cS(0))$ is infinite-dimensional.
\item We describe all homogeneous $\delta$-derivations of $\cS(q)$, $q\neq 0$, $\delta\notin\{0,1\}$ (Theorem \ref{thm:delta}).
Nontrivial homogeneous $\delta$-derivations exist only if $q\in2\mathbb Z^{*}$; they are the nonzero multiples of the odd derivation $D_6$
(for $\delta=\frac12$ the nontrivial $\frac12$-derivations are the maps $l\,\id+c\,D_6$ with $c\neq0$).
\item $\cS(q)$, $q\neq 0$, admits no nontrivial transposed $\delta$-Poisson superalgebra structures for $\delta\notin\{0,1\}$
(Theorem \ref{thm:tdp}). The reason is that the element $L_{0,-2q}$, which is responsible for the transposed $\delta$-Poisson structure on $\cB(q)$,
lies in the derived superalgebra of $\cS(q)$: $L_{0,-2q}=\frac1{2q}[G_{0,-2q},G_{0,0}]$.
\end{itemize}

\section{Preliminaries}

In this section, we recall some definitions and known results for studying derivations and quasi-derivations of Lie superalgebras.
All algebras and vector spaces are considered over the complex field.
For a set of vectors ${\rm S}$ we denote its linear span by $\langle {\rm S}\rangle$.
By $\delta_{i,j}$ we denote the standard Kronecker symbol.
For a superalgebra $L=L_{\bar0}\oplus L_{\bar1}$ we denote by $\hg(L)$ the set of homogeneous elements; if $|x|$ appears in a formula,
$x$ is assumed to be homogeneous of degree $|x|\in\mathbb Z_2$.
The letters $\theta,\mu$ denote elements of $\mathbb Z_2$.

\subsection{\normalsize Main definitions}

Let $\mathfrak{gl}_\theta(L)=\{D\in{\rm End}(L) : D(L_\mu)\subseteq L_{\mu+\theta}\ \text{for all}\ \mu\in\mathbb Z_2\}$.
Then $\mathfrak{gl}(L)=\mathfrak{gl}_{\bar0}(L)\oplus\mathfrak{gl}_{\bar1}(L)$ is a Lie superalgebra with the multiplication 
$[D_\theta,D_\mu]=D_\theta D_\mu-(-1)^{\theta\mu}D_\mu D_\theta$.

\begin{definition}
Let $L$ be a Lie superalgebra and $\delta\in\mathbb C$. An element $D\in\mathfrak{gl}_\theta(L)$ is called
\begin{itemize}
\item a \textit{homogeneous derivation of degree $\theta$} if
\begin{equation}\label{der}
D([x,y])=[D(x),y]+(-1)^{\theta|x|}[x,D(y)];
\end{equation}
\item a \textit{homogeneous generalized derivation of degree $\theta$} if there exist $D',D''\in\mathfrak{gl}_\theta(L)$ such that
$[D(x),y]+(-1)^{\theta|x|}[x,D'(y)]=D''([x,y])$;
\item a \textit{homogeneous quasi-derivation of degree $\theta$} if there exists $D'\in\mathfrak{gl}_\theta(L)$ such that
\begin{equation}\label{qder}
[D(x),y]+(-1)^{\theta|x|}[x,D(y)]=D'([x,y]);
\end{equation}
\item a \textit{homogeneous $\delta$-derivation of degree $\theta$} if $D([x,y])=\delta\big([D(x),y]+(-1)^{\theta|x|}[x,D(y)]\big)$,
\end{itemize}
where the identities are required for all $x\in\hg(L)$ and $y\in L$.
\end{definition}

We denote the corresponding spaces by $\Der_\theta(L)$, $\GDer_\theta(L)$, $\QDer_\theta(L)$ and $\Der_{\delta,\theta}(L)$, and put
$\Der(L)=\Der_{\bar0}(L)\oplus\Der_{\bar1}(L)$, etc. By $\Inn(L)=\ad(L)=\{\ad_x : x\in L\}$ we denote the inner derivations, $\ad_x(y)=[x,y]$.
If $\delta\neq 0$ and $D$ is a $\delta$-derivation, then $D$ is a quasi-derivation with $D'=\delta^{-1}D$.
Hence
$$\Inn(L)\subseteq\Der(L)\subseteq\QDer(L)\subseteq\GDer(L)\subseteq\mathfrak{gl}(L),\qquad
\Der_{\delta}(L)\subseteq\QDer(L)\ (\delta\neq0).$$
The identity map is a quasi-derivation with $D'=2\,\id$, and it is a $\frac12$-derivation.
If $Z(L)$ is the center of $L$, then every map in
$$\Ann^Q(L):=\big\{g\in \mathfrak{gl}_{\bar0}(L)\cup\mathfrak{gl}_{\bar1}(L) : g(L)\subseteq Z(L)\big\}$$
is a quasi-derivation with $D'=0$, since both sides of \eqref{qder} vanish.

\begin{definition}[see \cite{delta,deltu}]
Let $\delta\in\mathbb C$. A \textit{transposed $\delta$-Poisson superalgebra} is a triple $(P,\cdot,[\cdot,\cdot])$, where
$(P,[\cdot,\cdot])$ is a Lie superalgebra, $(P,\cdot)$ is a supercommutative associative superalgebra
(that is, $P_\alpha\cdot P_\beta\subseteq P_{\alpha+\beta}$ and $x\cdot y=(-1)^{|x||y|}y\cdot x$), and
\begin{equation}\label{tdp}
z\cdot[x,y]=\delta\big([z\cdot x,y]+(-1)^{|z||x|}[x,z\cdot y]\big).
\end{equation}
A Lie superalgebra $(L,[\cdot,\cdot])$ has a \textit{nontrivial} transposed $\delta$-Poisson structure if there is a nonzero product $\cdot$
such that $(L,\cdot,[\cdot,\cdot])$ is a transposed $\delta$-Poisson superalgebra.
\end{definition}

\begin{lemma}\label{lem:Lz}
Let $\cdot$ be a graded supercommutative product on a Lie superalgebra $L$ satisfying \eqref{tdp}, and let $z\in L_\theta$.
Then the left multiplication $\mathrm L_z\colon x\mapsto z\cdot x$ is a homogeneous $\delta$-derivation of degree $\theta$.
\end{lemma}

\begin{proof}
Since the product is graded, $\mathrm L_z\in\mathfrak{gl}_\theta(L)$, and \eqref{tdp} is exactly the $\delta$-derivation identity for $\mathrm L_z$.
\end{proof}

\subsection{\normalsize Block Lie algebras $\cB(q)$
and Block Lie superalgebras $\cS(q)$}

Xia, You, and Zhou in \cite{xyz12} defined the Block algebra $\cB(q)$ for a fixed complex number $q\in\mathbb C$ as an algebra with the basis
$\big\{L_{m,i}\mid m,i\in\mathbb Z\big\}$ and the following multiplication table
\begin{center}
$[L_{m,i},L_{n,j}]\ =\ \big(n(i+q)-m(j+q)\big)L_{m+n,i+j},\quad i,j,m,n\in\mathbb Z.$
\end{center}
They characterized the automorphism group and the derivation algebra of $\cB(q)$ for any $q\in\mathbb C$.

\begin{proposition}[see \cite{lgz18}]\label{prop:derB}
For the algebra $\cB(q)$ we have
$$\Der\big(\cB(q)\big)=\begin{cases}
\Inn\big(\cB(q)\big)\oplus\langle D_1,\ D_2,\ D_5\rangle, & \text{if}\quad q=0,\\
\Inn\big(\cB(q)\big)\oplus\langle D_1,\ D_3,\ D_4,\ D_5\rangle, & \text{if}\quad q\in\mathbb Z^*,\\
\Inn\big(\cB(q)\big)\oplus\langle D_1\rangle, & \text{if}\quad q\notin\mathbb Z,
\end{cases}$$
where
\begin{longtable}{lll lll lll}
$D_1(L_{m,i})$&$=$&$i L_{m,i},$ &
$D_2(L_{m,i})$&$=$&$m L_{m,i},$ &
$D_3(L_{m,i})$&$=$&$m L_{m,i-q},$\\
$D_4(L_{m,i})$&$=$&$(q+i)L_{m,i-q},$ &
$D_5(L_{0,-2q})$&$=$&$L_{0,-q}.$
\end{longtable}
\noindent In particular, $\ad_{L_{0,-q}}=0$ if $q\in\mathbb Z$, and $D_2\in\Inn\big(\cB(q)\big)$ if $q\neq0$.
\end{proposition}

Let $q\in\mathbb C$. The Lie superalgebra $\cS(q)=\cS(q)_{\bar0}\oplus\cS(q)_{\bar1}$ has $\cS(q)_{\bar0}=\cB(q)$, and the odd part
$\cS(q)_{\bar1}$ has the basis $\{G_{m,i}\mid m,i\in\mathbb Z\}$, with the remaining multiplications
\begin{align}
[L_{m,i},G_{n,j}]&=\Big(n(i+q)-m\big(j+\tfrac q2\big)\Big)G_{m+n,i+j},\label{LG}\\
[G_{m,i},G_{n,j}]&=2q\,L_{m+n,i+j}.\label{GG}
\end{align}
By super skew-symmetry, $[G_{n,j},L_{m,i}]=-[L_{m,i},G_{n,j}]$.
We write $X_{m,i}$ for $L_{m,i}$ or $G_{m,i}$, and $c(X_a,X_b)$ for the structure constant in $[X_a,X_b]=c(X_a,X_b)X'_{a+b}$,
where $X'$ is the basis vector of the appropriate parity. We write $\cS=\cS(q)$ and $\cS_\theta=\cS(q)_\theta$ when $q$ is fixed.
The superalgebra $\cS$ is also $\mathbb Z^2$-graded: $\deg X_{m,i}=(m,i)$.
One checks directly that $\cS(q)$ is a Lie superalgebra for every $q\in\mathbb C$ (super skew-symmetry and the super Jacobi identity;
the latter is equivalent to every $\ad_x$, $x\in\hg(\cS)$, being a homogeneous derivation of degree $|x|$; both are verified formally).

\medskip
\noindent\textbf{Notation.} For $D\in\mathfrak{gl}(\cS)$ we write $D(X_a)_{Y_c}$ for the coefficient of the basis
vector $Y_c$ in $D(X_a)$. Comparing the coefficients of one basis vector on both sides of \eqref{der} or \eqref{qder}
for a pair of basis vectors gives a scalar equation. All the functional equations below arise in this way.
We use the following elementary formulas many times:
\begin{equation}\label{useful}
[L_{0,0},X_{m,i}]=mqX_{m,i},\ [L_{0,\beta},X_{m,i}]=m(\beta+q)X_{m,i+\beta},\ [G_{0,\beta},L_{m,i}]=m\big(\beta+\tfrac q2\big)G_{m,i+\beta}.
\end{equation}

It follows from the multiplication in the algebra  $\cS(q)$ that 
$Z(\cS(q))=0$ if $q\notin\mathbb Z$; $Z(\cS(q))=\langle L_{0,-q}\rangle$ if $q\in\mathbb Z^{*}$; and
$Z(\cS(0))=\langle L_{0,0},G_{0,0}\rangle$.

\section{Derivations of Block Lie superalgebras}

In this section we describe $\Der(\cS(q))$ for every $q\in\mathbb C$. Recall that a derivation is an arbitrary linear map
satisfying \eqref{der}; no finiteness assumption is made.

\subsection{\normalsize Outer derivations}

\begin{definition}\label{def:maps}
We define the following linear maps on $\cS$ ($X\in\{L,G\}$; basis vectors that are not mentioned are sent to $0$).

\medskip
\begin{longtable}{lll}
\hline
\multicolumn{3}{l}{Even maps:}\\
\hline 
map & definition & used for\\
\hline
$D_1$ & $X_{m,i}\mapsto i\,X_{m,i}$ & all $q$\\
$D_2$ & $X_{m,i}\mapsto m\,X_{m,i}$ & $q=0$\\
$D_3$ & $X_{m,i}\mapsto m\,X_{m,i-q}$ & $q\in\mathbb Z^*$\\
$D_4$ & $L_{m,i}\mapsto (i+q)\,L_{m,i-q}$,\quad $G_{m,i}\mapsto \bigl(i+\frac q2\bigr)G_{m,i-q}$ & $q\in\mathbb Z^*$\\
$E_L$ & $L_{0,0}\mapsto L_{0,0}$ & $q=0$\\
$E_G$ & $G_{0,0}\mapsto G_{0,0}$ & $q=0$\\
$T$ & $L_{m,i}\mapsto 0$,\quad $G_{m,i}\mapsto G_{m,i}$ & $q=0$\\
\hline\\ \hline
\multicolumn{3}{l}{Odd maps:}\\
\hline   
map & definition & used for\\
\hline
$D_6$ & $G_{0,-\frac{3q}{2}}\mapsto L_{0,-q}$ & $q\in 2\mathbb Z^*$\\
$O_1$ & $L_{m,i}\mapsto i\,G_{m,i}$,\quad $G_{m,i}\mapsto 0$ & $q=0$\\
$O_2$ & $L_{m,i}\mapsto m\,G_{m,i}$,\quad $G_{m,i}\mapsto 0$ & $q=0$\\
$O_5$ & $L_{0,0}\mapsto G_{0,0}$ & $q=0$\\
$\Th$ & $L_{m,i}\mapsto 0$,\quad $G_{m,i}\mapsto L_{m,i}$ & $q=0$\\
$\Th_0$ & $G_{0,0}\mapsto L_{0,0}$ & $q=0$\\
\hline
\end{longtable}
\end{definition}

\begin{remark}
If $q\ne0$ then $D_2=\frac1q\ad_{L_{0,0}}$ is inner by \eqref{useful}. On the even part, $D_1,\dots,D_4$ restrict to the maps with
the same names in Proposition \ref{prop:derB}.
For $q=0$ the maps $E_L$ and $O_5$ are the even and odd versions of $D_5$ ($L_{0,0}\mapsto L_{0,0}$).
\end{remark}

\begin{lemma}\label{lem:outer-are-der}
Each map in Definition~\ref{def:maps} is a homogeneous derivation of $\cS(q)$ of the indicated
parity, for the values of $q$ indicated in the tables.
\end{lemma}

\begin{proof}
All the maps except $E_L,E_G,O_5,\Th_0,D_6$ have the form
$X_a\mapsto w(X_a)\,X'_{a+s}$ for a fixed shift $s\in\mathbb Z^2$ (and fixed parity change). For such a map the derivation identity on
the pair $(X_a,X_b)$ reduces to the scalar identity
\[
w(X_{a+b})\,c(X_a,X_b)=w(X_a)\,c(X'_{a+s},X_b)+(-1)^{\theta|X_a|}\,w(X_b)\,c(X_a,X'_{b+s}).
\]
For each map, this is checked case by case ($LL$, $LG$,   $GG$) by a direct polynomial
computation. For example, for $D_4$ ($q=k\in\mathbb Z^*$) and the pair $(L_{m,i},G_{n,j})$ we need
\[
\bigl(i+j+\tfrac k2\bigr)\bigl(n(i+k)-m(j+\tfrac k2)\bigr)
=(i+k)\bigl(ni-m(j+\tfrac k2)\bigr)+\bigl(j+\tfrac k2\bigr)\bigl(n(i+k)-m(j-\tfrac k2)\bigr),
\]
and for the pair $(G_{m,i},G_{n,j})$ we need $2k(i+j+k)=2k(i+\tfrac k2)+2k(j+\tfrac k2)$.

The maps $E_L,E_G,O_5,\Th_0,D_6$ send one basis vector $X_{a_0}$ to a vector $X'_{c_0}$
and all other basis vectors to $0$. Such a map is a derivation as soon as
(i) $X'_{c_0}$ is central, and (ii) $X_{a_0}$ does not occur in any multiplication of two basis
vectors with a nonzero coefficient.
\begin{itemize}
\item For $q=0$ ($E_L,E_G,O_5,\Th_0$): $L_{0,0}$ and $G_{0,0}$ are central, and neither of them occurs in a
multiplication of basis vectors, because $[X_a,X_b]$ with $a+b=0$ has coefficient $s(a,-a)=0$ (see \eqref{eq:B0} below) and
$[G,G]=0$.
\item For $q=2k\in2\mathbb Z^*$ and $D_6$: $L_{0,-q}$ is central. Moreover, $G_{0,-3k}$ does not occur in any multiplication
of basis vectors with a nonzero coefficient: since $[\cS_{\bar0},\cS_{\bar0}]+[\cS_{\bar1},\cS_{\bar1}]\subseteq\cS_{\bar0}$, it could only occur in
$[L_{m,i},G_{-m,-3k-i}]=-[G_{-m,-3k-i},L_{m,i}]$, and by \eqref{LG} the coefficient is
$-m(i+2k)-m(-3k-i+k)=0$. \qedhere
\end{itemize}
\end{proof}

\subsection{\normalsize Main theorem on derivations}

\begin{theorem}\label{thm:main}
Let $q\in\mathbb C$. The homogeneous derivations of $\cS(q)$ are the following, and all sums are direct
(i.e.\ the listed outer maps are linearly independent modulo inner derivations):
\begin{longtable}{lcl}
$\Der_{\bar0}(\cS(q))$&$=$&$\begin{cases}
\ad(\cS_{\bar0})\oplus\langle D_1\rangle, & q\notin\mathbb Z,\\
\ad(\cS_{\bar0})\oplus\langle D_1,D_3,D_4\rangle, & q\in\mathbb Z^*,\\
\ad(\cS_{\bar0})\oplus\langle D_1,D_2,E_L,E_G,T\rangle, & q=0,
\end{cases}$\\[8mm]
$\Der_{\bar1}(\cS(q))$&$=$&$\begin{cases}
\ad(\cS_{\bar1}), & q\notin2\mathbb Z,\\
\ad(\cS_{\bar1})\oplus\langle D_6\rangle, & q\in2\mathbb Z^*,\\
\ad(\cS_{\bar1})\oplus\langle O_1,O_2,O_5,\Th,\Th_0\rangle, & q=0.
\end{cases}$
\end{longtable}
\noindent In particular, the outer part $\Der(\cS(q))/\Inn(\cS(q))$ has dimension
$1$ if $q\notin\mathbb Z$, dimension $3$ if $q$ is an odd integer, dimension $3+1=4$ if $q\in2\mathbb Z^*$, and dimension $5+5=10$ if $q=0$.
\end{theorem}

\begin{remark}[Comparison with $\cB(q)$]\label{rem:compare}

\begin{enumerate}
\item For $q\in\mathbb Z^*$ the derivation $D_5\colon L_{0,-2q}\mapsto L_{0,-q}$ of the even part $\cB(q)$ does
\emph{not} extend to an even derivation of $\cS(q)$. The reason is that
$L_{0,-2q}=(2q)^{-1}[G_{0,0},G_{0,-2q}]$ lies in the derived superalgebra of $\cS(q)$.
\item Instead, a new \emph{odd} outer derivation $D_6$ appears, but only when $q$ is an even
integer. It exists because $G_{0,-3q/2}$ is not a multiplication of basis vectors and
$L_{0,-q}$ is central.
\item For $q=0$ we have $[\cS_{\bar1},\cS_{\bar1}]=0$, and $\cS(0)=\cB(0)\ltimes M$, where the odd part $M$ is an
abelian ideal isomorphic to the adjoint $\cB(0)$-module. This produces the additional maps
$E_G$, $T$, $O_1$, $O_2$, $O_5$, $\Th$, $\Th_0$.
\end{enumerate}
\end{remark}

The remainder of this section proves Theorem~\ref{thm:main}. By
Lemma~\ref{lem:outer-are-der}, and since inner derivations are derivations, every map on
the right-hand side is a derivation. It remains to prove that every derivation has the stated form
(Subsections \ref{sec:qne0} and \ref{sec:q0}) and that the sums are direct (Subsection \ref{sec:direct}).

\subsection{\normalsize Derivations of $\cS(q)$ for $q\neq0$}\label{sec:qne0}

\begin{lemma}\label{lem:reduce}
Let $q\ne0$ and $D\in\Der_\theta(\cS)$. There is $x\in\cS_\theta$ such that
$\tilde D=D-\ad_x$ preserves the first index: $\tilde D(X_{m,i})\in\langle Y_{m,j}\mid Y\in\{L,G\},\ j\in\mathbb Z\rangle$.
\end{lemma}

\begin{proof}
Write $D(L_{0,0})=\sum\limits_{\gamma} c_\gamma X_\gamma$ (a finite sum), and put
\begin{center}
$x:\ =\ -\sum_{\gamma=(m_\gamma,i_\gamma),\ m_\gamma\ne0}\frac{c_\gamma}{m_\gamma q}\,X_\gamma\in\cS_\theta .$
\end{center}
Since $[L_{0,0},X_{m,i}]=mq\,X_{m,i}$, the derivation identity for the pair $(L_{0,0},X_b)$, $b=(m_b,i_b)$,
read at a basis vector $X_c$ with $m_c\ne m_b$, gives
\begin{center}
$(m_b-m_c)\,q\,D(X_b)_{X_c}=[D(L_{0,0}),X_b]_{X_c}.$
\end{center}
The same identity holds for the derivation $\ad_x$. Moreover
\begin{center}
$\ad_x(L_{0,0})=-[L_{0,0},x]=\sum_{m_\gamma\ne0}c_\gamma X_\gamma$,
\end{center} so $D(L_{0,0})-\ad_x(L_{0,0})$ has only components with
first index $0$. Hence $[D(L_{0,0})-\ad_x(L_{0,0}),X_b]$ has only components with first index $m_b$, and
therefore $\tilde D(X_b)_{X_c}=0$ whenever $m_c\ne m_b$.
\end{proof}

Consequently $\tilde D$ decomposes into its $\mathbb Z^2$-homogeneous components of degree $(0,\beta)$, $\beta\in\mathbb Z$.
For fixed $\beta$ we study the corresponding coefficients.

\medskip
\noindent\textbf{Even derivations.}
Let $\tilde D$ be as in Lemma~\ref{lem:reduce} with $\theta=\bar0$, and for $\beta\in\mathbb Z$ put
\[
f(m,i)=\tilde D(L_{m,i})_{L_{m,i+\beta}},\qquad g(m,i)=\tilde D(G_{m,i})_{G_{m,i+\beta}}.
\]
(For simplicity we suppress $\beta$ from the notation.) The pairs $(G_{m,i},G_{n,j})$ and
$(L_{m,i},L_{n,j})$ give, after cancelling $2q\ne0$ in the first,
\begin{align}
f(m+n,i+j)&\ =\ g(m,i)+g(n,j), \label{eq:evC}\\
\bigl(n(i+q)-m(j+q)\bigr)f(m+n,i+j)&\ =\ f(m,i)\bigl(n(i+\beta+q)-m(j+q)\bigr)\notag\\
&\ \ \quad+\bigl(n(i+q)-m(j+\beta+q)\bigr)f(n,j). \label{eq:evA}
\end{align}

\begin{lemma}\label{lem:even-solve}
Every solution of \eqref{eq:evC}--\eqref{eq:evA} has the form
\[
g(m,i)=Am+Bi+C\qquad\text{and}\qquad f(m,i)=Am+Bi+2C,
\]
where $A,B,C\in\mathbb C$ satisfy
\[
\beta B+2C=0\qquad\text{and}\qquad C(\beta+q)=0.
\]
\end{lemma}

\begin{proof}
Putting $n=j=0$ in \eqref{eq:evC} gives $f(m,i)=g(m,i)+g(0,0)$. Hence
\begin{center}
$g(m+n,i+j)+g(0,0)=g(m,i)+g(n,j)$,
\end{center} so $g-g(0,0)$ is additive on $\mathbb Z^2$, and therefore
$g(m,i)=Am+Bi+C$ with $C=g(0,0)$. Substituting into \eqref{eq:evA} with
$(m,i,n,j)=(0,0,1,0)$ gives $2C(\beta+q)=0$. With $(m,i,n,j)=(0,1,1,0)$ it gives
$(1+q)B=(B+2C)(1+\beta+q)$, that is, $\beta B+2C+2C(\beta+q)=0$.
\end{proof}

\begin{proposition}\label{prop:even-qne0}
Let $q\ne0$ and $D\in\Der_{\bar0}(\cS(q))$. Then
\begin{itemize}
\item $D\in\ad(\cS_{\bar0})+\langle D_1\rangle$ if $q\notin\mathbb Z$;
\item $D\in\ad(\cS_{\bar0})+\langle D_1,D_3,D_4\rangle$ if $q\in\mathbb Z^*$.
\end{itemize}
\end{proposition}

\begin{proof}
Let $\tilde D=D-\ad_x$ be as in Lemma~\ref{lem:reduce}. Fix $\beta$ and let $A_\beta,B_\beta,C_\beta$ be
given by Lemma~\ref{lem:even-solve}. Recall \eqref{useful}:
$[L_{0,\beta},X_{m,i}]=m(\beta+q)X_{m,i+\beta}$.
\begin{itemize}
\item If $\beta+q\ne0$ and $\beta\ne0$, then $C_\beta=0$ and $B_\beta=0$, so $f=g=A_\beta m$. The
degree-$(0,\beta)$ component of $\tilde D$ is therefore $\frac{A_\beta}{\beta+q}\ad_{L_{0,\beta}}$.
\item If $\beta=0$, then $C_0=0$ (as $q\neq0$) and $f=g=A_0m+B_0i$. The degree-$(0,0)$ component is
$\frac{A_0}{q}\ad_{L_{0,0}}+B_0D_1$.
\item If $\beta+q=0$, then $q=k\in\mathbb Z^*$ and $\beta=-k$, and $C=\frac k2 B$. Hence
$f=A m+B(i+k)$ and $g=A m+B(i+\frac k2)$, so the degree-$(0,-k)$ component is $A_{-k}D_3+B_{-k}D_4$.
\end{itemize}
Put $Y_\beta=A_\beta/(\beta+q)$ for $\beta+q\ne0$, and $Y_\beta=0$ otherwise. Since $A_\beta=g(1,0)-g(0,0)$,
$Y_\beta\ne0$ forces $\beta$ to lie in the finite set of second indices occurring in
$\tilde D(G_{1,0})$ or $\tilde D(G_{0,0})$. Hence $y=\sum_\beta Y_\beta L_{0,\beta}$ is a finite sum, and
\[
D=\ad_{x+y}+B_0D_1\ \bigl(+A_{-k}D_3+B_{-k}D_4\ \text{ if } q=k\in\mathbb Z^*\bigr). \qedhere
\]
\end{proof}

\medskip
\noindent\textbf{Odd derivations.}
Let $\tilde D$ be as in Lemma~\ref{lem:reduce} with $\theta=\bar1$, and for fixed $\beta$ put
\[
f(m,i)=\tilde D(L_{m,i})_{G_{m,i+\beta}},\qquad g(m,i)=\tilde D(G_{m,i})_{L_{m,i+\beta}}.
\]
The pairs $(L,L)$, $(L,G)$ and $(G,G)$ give
\begin{align}
\bigl(n(i+q)-m(j+q)\bigr)f(m+n,i+j)&\ =\ f(m,i)\bigl(n(i+\beta+\tfrac q2)-m(j+q)\bigr)\notag\\
&\ \ \quad+\bigl(n(i+q)-m(j+\beta+\tfrac q2)\bigr)f(n,j),\label{eq:odA}\\
\bigl(n(i+q)-m(j+\tfrac q2)\bigr)g(m+n,i+j)&\ =\ 2q\,f(m,i)+\bigl(n(i+q)-m(j+\beta+q)\bigr)g(n,j),\label{eq:odB}\\
2q\,f(m+n,i+j)&\ =\ g(m,i)\bigl(n(i+\beta+q)-m(j+\tfrac q2)\bigr)\notag\\
&\ \ \quad-\bigl(n(i+\tfrac q2)-m(j+\beta+q)\bigr)g(n,j).\label{eq:odC}
\end{align}

\begin{lemma}\label{lem:odd-solve}
Let $q\ne0$. Every solution of \eqref{eq:odA}--\eqref{eq:odC} satisfies
\begin{enumerate}
\item $f(m,i)=a\,m$ with $a=f(1,0)$;
\item $2q\,a=g(1,0)\bigl(\beta+\frac q2\bigr)$;
\item $g(m,i)=g(1,0)$ for all $(m,i)$, except possibly for $m=0$ and $i$ satisfying both
$i+\beta+q=0$ and $i+\frac{3q}{2}=0$.
\end{enumerate}
\end{lemma}

\begin{proof}
(1) Put $m=n=0$ in \eqref{eq:odC} to get $f(0,i)=0$. Now let $m\ne0$. With $n=0$, \eqref{eq:odA} gives
$m(j+q)\bigl(f(m,i+j)-f(m,i)\bigr)=0$. So $f(m,\cdot)$ is invariant under every shift $j\ne-q$, and
hence constant (every shift is a composite of two admissible ones). Comparing \eqref{eq:odA} for $(m,1,n,0)$ and $(m,0,n,0)$ gives
$n\bigl(f(m+n,0)-f(m,0)-f(n,0)\bigr)=0$. So $m\mapsto f(m,0)$ is additive, and $f(m,i)=am$.

(3), (2) With $m=0$, \eqref{eq:odB} and $f(0,\cdot)=0$ give $n(i+q)(g(n,i+j)-g(n,j))=0$. Hence $g(n,\cdot)$ is constant for $n\ne0$.
With $n=0$ and $m\ne0$, \eqref{eq:odB} gives
\begin{equation}\label{eq:star}
-\bigl(j+\tfrac q2\bigr)g(m,0)=2qa-(j+\beta+q)\,g(0,j).
\end{equation}
Comparing \eqref{eq:star} at $j=0$ for $m$ and for $m=1$ gives $g(m,0)=g(1,0)$ for all $m\ne0$, since $q/2\ne0$.
Then \eqref{eq:odB} with $(1,0,1,0)$ gives (2). Finally, for $m=0$: if $i+\beta+q\ne0$, then \eqref{eq:star}
together with (2) gives $g(0,i)=g(1,0)$. If $i+\frac{3q}{2}\ne0$, then \eqref{eq:odB} with $(1,0,-1,i)$
together with (2) gives the same conclusion.
\end{proof}

\begin{proposition}\label{prop:odd-qne0}
Let $q\ne0$ and $D\in\Der_{\bar1}(\cS(q))$. Then
\begin{itemize}
\item $D\in\ad(\cS_{\bar1})$ if $q\notin2\mathbb Z$;
\item $D\in\ad(\cS_{\bar1})+\langle D_6\rangle$ if $q\in2\mathbb Z^*$.
\end{itemize}
\end{proposition}

\begin{proof}
Let $\tilde D=D-\ad_x$ be as in Lemma~\ref{lem:reduce}. We have
$[G_{0,\beta},L_{m,i}]=m(\beta+\frac q2)G_{m,i+\beta}$ and $[G_{0,\beta},G_{m,i}]=2q\,L_{m,i+\beta}$.
By Lemma~\ref{lem:odd-solve}, the degree-$(0,\beta)$ component of $\tilde D$ coincides with
$Y_\beta\ad_{G_{0,\beta}}$, where $Y_\beta=g_\beta(1,0)/(2q)$, except possibly at the single
coefficient $g_\beta(0,i)$ with $i+\beta+q=0=i+\frac{3q}{2}$. Such an exceptional coefficient
requires $q=2\beta$, i.e.\ $q=2k\in2\mathbb Z^*$ with $\beta=k$ and $i=-3k$. At this coefficient the
difference is $\bigl(g_k(0,-3k)-g_k(1,0)\bigr)D_6$. As before, $Y_\beta\ne0$ only for the finitely many $\beta$
occurring in $\tilde D(G_{1,0})$, so $y=\sum Y_\beta G_{0,\beta}\in\cS_{\bar1}$. Hence
$D=\ad_{x+y}$ if $q\notin2\mathbb Z$, and $D=\ad_{x+y}+\bigl(g_k(0,-3k)-g_k(1,0)\bigr)D_6$ if $q=2k$.
\end{proof}

\subsection{\normalsize Derivations of $\cS(0)$}\label{sec:q0}

For $q=0$ we have $[L_{0,0},\cdot\,]=0$, so Lemma~\ref{lem:reduce} is not available. Instead we
use the structure $\cS(0)=\cB(0)\ltimes M$. Write $\alpha=(m,i)$, $L_\alpha=L_{m,i}$, $G_\alpha=G_{m,i}$ and
\begin{equation}\label{eq:B0}
s(\alpha,\beta)=\beta_1\alpha_2-\alpha_1\beta_2\qquad(\alpha,\beta\in\mathbb Z^2).
\end{equation}
Then $s$ is bilinear and alternating, and
\[
[L_\alpha,L_\beta]=s(\alpha,\beta)L_{\alpha+\beta},\qquad
[L_\alpha,G_\beta]=s(\alpha,\beta)G_{\alpha+\beta},\qquad
[G_\alpha,G_\beta]=0.
\]

For a function $F\colon\mathbb Z^2\to\mathbb C$ and $\sigma\in\mathbb Z^2$ consider the equations
\begin{align}
s(\alpha,\beta)F(\alpha+\beta)&=F(\alpha)\,s(\alpha+\sigma,\beta)+s(\alpha,\beta+\sigma)F(\beta),\tag{E$_\sigma$}\label{eq:E}\\
s(\alpha,\beta)H(\alpha+\beta)&=s(\alpha,\beta+\sigma)H(\beta).\tag{Z$_\sigma$}\label{eq:Z}
\end{align}
Equation (E$_\sigma$) says that $L_\alpha\mapsto F(\alpha)L_{\alpha+\sigma}$ is a derivation of $\cB(0)$ of degree $\sigma$;
equation (Z$_\sigma$) says that the map $H$ commutes with the action of $\cB(0)$ in degree $\sigma$.

\begin{lemma}\label{lem:der-ne}
Let $\sigma\ne0$, let $F$ satisfy \textup{(E$_\sigma$)}, and choose $\rho$ with $s(\sigma,\rho)\ne0$. Then
$F(\alpha)=t\,s(\sigma,\alpha)$ for all $\alpha$, where $t=F(\rho)/s(\sigma,\rho)$. In other words, the component is $t\,\ad_{L_\sigma}$.
\end{lemma}

\begin{proof}
The function $\alpha\mapsto s(\sigma,\alpha)$ solves (E$_\sigma$), so $G=F-t\,s(\sigma,\cdot)$ solves
(E$_\sigma$) and satisfies $G(\rho)=0$. We show that $G\equiv0$.
\begin{enumerate}
\item Combining (E$_\sigma$) at $(\rho,-\rho)$, $(-\rho,\sigma)$ and $(\rho,\sigma-\rho)$ gives
$3s(\sigma,\rho)G(\sigma)=0$, so $G(\sigma)=0$.
\item (E$_\sigma$) at $(x,\sigma)$ gives $s(x,\sigma)G(x+\sigma)=s(x,\sigma)G(x)$. So $G(x+\sigma)=G(x)$ whenever $s(\sigma,x)\neq0$.
\item If $s(\sigma,\alpha),s(\sigma,\beta),s(\sigma,\alpha+\beta)$ are all nonzero, then comparing (E$_\sigma$) at
$(\alpha,\beta)$ and $(\alpha,\beta+\sigma)$, using (2), gives $G(\alpha+\beta)=G(\alpha)+G(\beta)$. Then (E$_\sigma$)
at $(\alpha,\beta)$ becomes $G(\alpha)s(\sigma,\beta)=G(\beta)s(\sigma,\alpha)$.
\item If $s(\sigma,\alpha)\ne0$, apply (3) with $\beta=\rho$, or with $\beta=2\rho$ when $s(\sigma,\alpha+\rho)=0$ (first obtaining $G(2\rho)=0$ from (3) with $\alpha=\rho$, $\beta=\rho$). This gives $G(\alpha)=0$.
\item (E$_\sigma$) at $(\rho,0)$ gives $s(\rho,\sigma)G(0)=0$, so $G(0)=0$.
\item Let $s(\sigma,\alpha)=0$. If also $s(\rho,\alpha)=0$, then $\alpha=0$ because $\sigma,\rho$ are linearly independent. Otherwise
(E$_\sigma$) at $(\rho,\alpha-\rho)$ gives $s(\rho,\alpha)G(\alpha)=0$, since $G(\rho)=0$ and $G(\alpha-\rho)=0$ by (4)
(note $s(\sigma,\alpha-\rho)=-s(\sigma,\rho)\ne0$). \qedhere
\end{enumerate}
\end{proof}

\begin{lemma}\label{lem:der-zero}
If $F$ satisfies \textup{(E$_0$)}, then $F(\alpha)=a\,\alpha_2+b\,\alpha_1$ for all $\alpha\ne0$, where
$a=F(0,1)$ and $b=F(1,0)$. The value $F(0)$ is arbitrary.
\end{lemma}

\begin{proof}
Let $G(\alpha)=F(\alpha)-a\alpha_2-b\alpha_1$. Then $G$ satisfies (E$_0$) and $G(1,0)=G(0,1)=0$, and
(E$_0$) says $G(\alpha+\beta)=G(\alpha)+G(\beta)$ whenever $s(\alpha,\beta)\ne0$. Since $s((m,1),(1,0))=1$, we get
$G(m+1,1)=G(m,1)+G(1,0)=G(m,1)$ for all $m\in\mathbb Z$, hence $G(m,1)=G(0,1)=0$ for all $m$.
Similarly, $s((1,i),(0,1))=-1$ gives $G(1,i+1)=G(1,i)$, hence $G(1,i)=G(1,0)=0$ for all $i$.
Next, $s((-2,1),(1,-1))=-1$ gives $G(-1,0)=G(-2,1)+G(1,-1)=0$.
For $j\ne0$ we have $s((1,j),(-1,0))=-j\neq0$, so $G(0,j)=G(1,j)+G(-1,0)=0$;
for $n\neq0$ we have $s((n,1),(0,-1))=n\neq0$, so $G(n,0)=G(n,1)+G(0,-1)=0$.
Finally, for $mi\ne0$ we have $s((m,0),(0,i))=-mi\neq0$, so $G(m,i)=G(m,0)+G(0,i)=0$.
\end{proof}

\begin{lemma}\label{lem:cent-ne}
If $\sigma\ne0$ and $H$ satisfies \textup{(Z$_\sigma$)}, then $H\equiv0$.
\end{lemma}

\begin{proof}
(Z$_\sigma$) at $(-y,y)$ gives $s(\sigma,y)H(y)=0$. Choose $\rho$ with $s(\sigma,\rho)\ne0$. Then
(Z$_\sigma$) at $(\rho,0)$ gives $H(0)=0$. If $s(\sigma,\alpha)=0$ and $\alpha\ne0$, then $s(\rho,\alpha)\ne0$, and
(Z$_\sigma$) at $(\rho,\alpha-\rho)$ gives $s(\rho,\alpha)H(\alpha)=0$.
\end{proof}

\begin{lemma}\label{lem:cent-zero}
If $H$ satisfies \textup{(Z$_0$)}, then $H$ is constant on $\mathbb Z^2\setminus\{0\}$.
\end{lemma}

\begin{proof}
(Z$_0$) at $(x-\beta,\beta)$ gives $H(x)=H(\beta)$ whenever $s(x,\beta)\ne0$. Any two nonzero vectors can be connected by a chain of such steps via $(1,0)$ and $(0,1)$.
\end{proof}

\begin{proposition}\label{prop:even-q0}
$\Der_{\bar0}(\cS(0))=\ad(\cS_{\bar0})+\langle D_1,D_2,E_L,E_G,T\rangle$.
\end{proposition}

\begin{proof}
Let $D$ be an even derivation. For $\sigma\in\mathbb Z^2$ put \begin{center}
    $f_\sigma(\alpha)=D(L_\alpha)_{L_{\alpha+\sigma}}$ and
$g_\sigma(\alpha)=D(G_\alpha)_{G_{\alpha+\sigma}}$.
\end{center} The pair $(L_\alpha,L_\beta)$ shows that $f_\sigma$ satisfies
(E$_\sigma$). The pair $(L_\alpha,G_\beta)$ gives
\[
s(\alpha,\beta)g_\sigma(\alpha+\beta)=f_\sigma(\alpha)s(\alpha+\sigma,\beta)+s(\alpha,\beta+\sigma)g_\sigma(\beta).
\]
Subtracting, we see that $h_\sigma=g_\sigma-f_\sigma$ satisfies (Z$_\sigma$).

For $\sigma\ne0$, Lemmas~\ref{lem:der-ne} and~\ref{lem:cent-ne} give $f_\sigma=g_\sigma=t_\sigma s(\sigma,\cdot)$. So the
degree-$\sigma$ component of $D$ is $t_\sigma\ad_{L_\sigma}$. Moreover $t_\sigma\ne0$ forces $f_\sigma(1,0)\ne0$ or
$f_\sigma(0,1)\ne0$, which happens only for finitely many $\sigma$. Hence $x=\sum_\sigma t_\sigma L_\sigma\in\cS_{\bar0}$.

For $\sigma=0$, Lemmas~\ref{lem:der-zero} and~\ref{lem:cent-zero} give
$f_0(\alpha)=a\alpha_2+b\alpha_1$ and $g_0(\alpha)=a\alpha_2+b\alpha_1+t$ for $\alpha\ne0$.
The values $f_0(0)=c$ and $g_0(0)$ are arbitrary. Hence
\[
D=\ad_x+aD_1+bD_2+cE_L+\bigl(g_0(0)-t\bigr)E_G+tT. \qedhere
\]
\end{proof}

\begin{proposition}\label{prop:odd-q0}
$\Der_{\bar1}(\cS(0))=\ad(\cS_{\bar1})+\langle O_1,O_2,O_5,\Th,\Th_0\rangle$.
\end{proposition}

\begin{proof}
Let $D$ be odd, and put $u_\sigma(\alpha)=D(L_\alpha)_{G_{\alpha+\sigma}}$ and $v_\sigma(\alpha)=D(G_\alpha)_{L_{\alpha+\sigma}}$.
The pair $(L_\alpha,L_\beta)$ shows that $u_\sigma$ satisfies (E$_\sigma$). The pair $(L_\alpha,G_\beta)$, together with $[G,G]=0$, shows that $v_\sigma$ satisfies (Z$_\sigma$).
Note that $[G_\sigma,L_\alpha]=s(\sigma,\alpha)G_{\sigma+\alpha}$ and $[G_\sigma,G_\alpha]=0$.

For $\sigma\ne0$: $u_\sigma=t_\sigma s(\sigma,\cdot)$ and $v_\sigma=0$, so the component is $t_\sigma\ad_{G_\sigma}$. Again only finitely many
$t_\sigma$ are nonzero, so $y=\sum t_\sigma G_\sigma\in\cS_{\bar1}$.

For $\sigma=0$: $u_0(\alpha)=a\alpha_2+b\alpha_1$ and $v_0(\alpha)=d$ for $\alpha\ne0$, while $u_0(0)=c$ and $v_0(0)$ are arbitrary.
Hence
\[
D=\ad_y+aO_1+bO_2+cO_5+d\,\Th+\bigl(v_0(0)-d\bigr)\Th_0. \qedhere
\]
\end{proof}

\subsection{\normalsize Directness of the sums}\label{sec:direct}

Suppose that $\ad_x+\sum_i c_iD_i=0$ with $x=\sum_\gamma x_\gamma X_\gamma\in\cS_\theta$ and $D_i$ the outer maps of the
relevant case. Since $[X_\gamma,X_a]$ is a multiple of $X'_{a+\gamma}$, the coefficient of $X'_{a+\gamma}$ in
$\ad_x(X_a)$ equals $x_\gamma$ times one structure constant. We read off the relation at a few well-chosen
coefficients. In each case the contribution of $\ad_x$ is either $0$ (because the structure constant vanishes) or
is determined by a second coefficient.
\begin{itemize}
\item $q\notin\mathbb Z$: the coefficient of $G_{0,1}$ in the image of $G_{0,1}$ gives the coefficient of $D_1$.
\item $q=k\in\mathbb Z^*$: use $G_{0,1}\to G_{0,1}$ (for $D_1$), $L_{1,-k}\to L_{1,-2k}$ (for $D_3$; here
$[L_{0,-k},L_{1,-k}]=0$), and $L_{0,1-k}\to L_{0,1-2k}$ (for $D_4$).
\item $q=0$, even: use $L_{0,1}\to L_{0,1}$ ($D_1$), $L_{1,0}\to L_{1,0}$ ($D_2$), $L_{0,0}\to L_{0,0}$ ($E_L$),
$G_{1,0}\to G_{1,0}$ ($T$, after $D_2$), and $G_{0,0}\to G_{0,0}$ ($E_G$, after $T$).
\item $q=2k\in2\mathbb Z^*$, odd: the coefficient of $L_{0,-2k}$ in the image of $G_{0,-3k}$ is
$2q\,x_{G_{0,k}}+c$, while the coefficient of $L_{0,k}$ in the image of $G_{0,0}$ is $2q\,x_{G_{0,k}}$. Hence $c=0$.
\item $q=0$, odd: use $L_{0,1}\to G_{0,1}$ ($O_1$), $L_{1,0}\to G_{1,0}$ ($O_2$), $L_{0,0}\to G_{0,0}$ ($O_5$),
$G_{1,0}\to L_{1,0}$ ($\Th$; recall $[G,G]=0$), and $G_{0,0}\to L_{0,0}$ ($\Th_0$).
\end{itemize}
This completes the proof of Theorem~\ref{thm:main}. \hfill$\Box$

\section{Quasi-derivations of Block Lie superalgebras}

For $q\neq 0$ we set
\begin{gather*}
\Ann^Q(\cS(q))=\Ann^Q_{\bar0}\oplus\Ann^Q_{\bar1},\\
\Ann^Q_{\theta}=\big\{g\in\mathfrak{gl}(\cS(q)) : g(\cS(q))\subseteq\langle L_{0,-q}\rangle,\ g(\cS(q)_{\theta+\bar1})=0\big\}
\end{gather*}
if $q\in\mathbb Z^*$, and $\Ann^Q(\cS(q))=0$ if $q\notin\mathbb Z$. $\Ann^Q_\theta$ consists of the maps in
$\mathfrak{gl}_\theta(\cS(q))$ with values in the center.
Note that $\Ann^Q_\theta$ is infinite-dimensional: the values $g(X_{m,i})\in\langle L_{0,-q}\rangle$ for $X_{m,i}\in\cS(q)_\theta$ can be chosen arbitrarily.
In particular, $\Ann^Q_{\bar0}$ contains the maps $g_{d,k}\colon L_{m,i}\mapsto\delta_{d,m}\delta_{k,i}L_{0,-q}$,
and $\Ann^Q_{\bar1}$ contains the maps $\tilde g_{d,k}\colon G_{m,i}\mapsto\delta_{d,m}\delta_{k,i}L_{0,-q}$.
For $q\in2\mathbb Z^*$ we have $D_6=\tilde g_{0,-\frac{3q}{2}}\in\Ann^Q_{\bar1}$.

\begin{remark}
For $q\in\mathbb Z^*$ the derivation $D_5\colon L_{0,-2q}\mapsto L_{0,-q}$ of $\cB(q)$ does not extend to a derivation of $\cS(q)$
(Remark \ref{rem:compare}). It is, however, an element of $\Ann^Q_{\bar0}(\cS(q))$ and hence a quasi-derivation.
\end{remark}

Let $D\in\QDer_\theta(\cS(q))$ and let $D'$ be as in \eqref{qder}.
Comparing coefficients in \eqref{qder} for $x=X_a$, $y=X_b$, we obtain the main identity
\begin{equation}\label{mainidentity}
c(X_a,X_b)\,D'(X'_{a+b})_{Y_c}=D(X_a)_{Y_{c-b}}\,c(Y_{c-b},X_b)+(-1)^{\theta|X_a|}c(X_a,Y_{c-a})\,D(X_b)_{Y_{c-a}},
\end{equation}
where $Y_{c-b}$ and $Y_{c-a}$ are the basis vectors (of the appropriate parities) with indices $c-b$ and $c-a$.
Note that if $(D,D')$ satisfies \eqref{qder} and $x\in\cS(q)_\theta$, then so does $(D-\ad_x,D'-\ad_x)$.

\subsection{\normalsize Quasi-derivations of $\cS(q)$ for $q\notin\mathbb Z$}

\begin{theorem}\label{thm:qnotZ}
Let $q\notin\mathbb Z$. Then
\begin{align*}
\QDer_{\bar0}(\cS(q))&=\ad(\cS(q)_{\bar0})\oplus\langle D_1,\ \id\rangle=\Der_{\bar0}(\cS(q))\oplus\langle\id\rangle,\\
\QDer_{\bar1}(\cS(q))&=\ad(\cS(q)_{\bar1})=\Der_{\bar1}(\cS(q)).
\end{align*}
\end{theorem}

\subsection{\normalsize Quasi-derivations of $\cS(q)$ for $q\in\mathbb Z^*$}

\begin{theorem}\label{thm:qZ}
Let $q\in\mathbb Z^*$. Then
\begin{align*}
\QDer_{\bar0}(\cS(q))&=\ad(\cS(q)_{\bar0})\oplus\langle D_1,D_3,D_4,\id\rangle\oplus\Ann^Q_{\bar0}(\cS(q)),\\
\QDer_{\bar1}(\cS(q))&=\ad(\cS(q)_{\bar1})\oplus\Ann^Q_{\bar1}(\cS(q)).
\end{align*}
\end{theorem}

\subsection{\normalsize Proof of Theorems \ref{thm:qnotZ} and \ref{thm:qZ}}

Throughout, $q\neq0$. All maps on the right-hand sides are quasi-derivations: derivations with $D'=D$ (Theorem \ref{thm:main}),
the identity with $D'=2\id$, and elements of $\Ann^Q$ with $D'=0$. Let $(D,D')$ satisfy \eqref{qder}, $D,D'\in\mathfrak{gl}_\theta(\cS(q))$.

\medskip
\noindent\textbf{Step 1: normalization.}
Since $[L_{0,0},X_{m,i}]=mq\,X_{m,i}$, the element $x=-\sum_{m\neq0}(mq)^{-1}D(L_{0,0})_m\in\cS(q)_\theta$,
where $D(L_{0,0})_m$ is the part of $D(L_{0,0})$ with first index $m$, satisfies
$(D-\ad_x)(L_{0,0})\in\langle X_{0,j} : j\in\mathbb Z\rangle$. Replacing $(D,D')$ by $(D-\ad_x,D'-\ad_x)$, we may assume that
$D(L_{0,0})$ involves only basis vectors with first index $0$.

\medskip
\noindent\textbf{Step 2: graded components.}
For $\alpha=(\alpha_1,\alpha_2)\in\mathbb Z^2$ let $f(m,i)$ and $g(m,i)$ be the coefficients of $X_{m+\alpha_1,i+\alpha_2}$ (of the appropriate parity)
in $D(L_{m,i})$ and $D(G_{m,i})$, and let $f',g'$ be the corresponding coefficients for $D'$.
Identity \eqref{mainidentity} splits into a system of functional equations for $f,g,f',g'$ (for the pairs $LL$, $LG$, $GL$, $GG$).
Taking $X_a=L_{0,0}$ gives
\begin{equation}\label{R}
nq\,f'(n,j)=u_\alpha\,c_\alpha(n,j)+(n+\alpha_1)q\,f(n,j),
\end{equation}
and similarly for $g,g'$, where $u_\alpha$ is the coefficient of $X_\alpha$ in $D(L_{0,0})$ and $c_\alpha(n,j)$ is the corresponding structure constant.

\medskip
\noindent\textbf{Step 3: components with $\alpha_1\neq0$.}
By Step 1, $u_\alpha=0$, so \eqref{R} gives 
\begin{center}
$mf'(m,i)=(m+\alpha_1)f(m,i)$ and $mg'(m,i)=(m+\alpha_1)g(m,i)$.
\end{center}In particular $f(0,\cdot)=g(0,\cdot)=0$.

\emph{Even case.} The $GG$-equation gives $f'(m+n,i+j)=g(m,i)+g(n,j)$. Taking $m=0$ and using $g(0,\cdot)=0$ we get $f'(n,i+j)=g(n,j)$ for all $i$,
so $g(n,j)=f'(n,j)$ does not depend on $j$ and is additive in $n$; hence $g(n,j)=f'(n,j)=\gamma n$
for a constant $\gamma$. At $m=-\alpha_1$ the relation $mf'=(m+\alpha_1)f$ gives $\alpha_1^2\gamma=0$, so $\gamma=0$.
Then $g=0$, $f'=0$ and $f(m,\cdot)=0$ for $m\neq-\alpha_1$. The $LL$-equation with $(m,i,n,j)=(\alpha_1,0,-\alpha_1,j)$ gives
$\alpha_1(j+\alpha_2+q)f(-\alpha_1,j)=0$. Hence the only coefficient that may survive is the coefficient of $L_{0,-q}$ in $D(L_{-\alpha_1,j})$ with $j+\alpha_2=-q$,
which can occur only if $q\in\mathbb Z$.

\emph{Odd case.} The $LG$-equation with $m=0$ shows that $g(n,\cdot)$ is constant for $n\neq-\alpha_1$. With $n=0$ it gives
$2qf(m,i)=-(j+\frac q2)(m+\alpha_1)g(m,i+j)$ for all $j$. Hence $f=0$ and $g(m,\cdot)=0$ for $m\neq-\alpha_1$.
The $GG$-equation at $(m,i,n,j)=(-\alpha_1,j,2\alpha_1,0)$ leaves only the coefficient of $L_{0,-q}$ in $D(G_{-\alpha_1,-q-\alpha_2})$.

\medskip
\noindent\textbf{Step 4: components with $\alpha=(0,\beta)$.}

\emph{Even case.} Now \eqref{R} gives $f'=f+\varepsilon$ and $g'=g+\varepsilon$ for $m\neq0$, where $\varepsilon=u_\beta(\beta+q)/q$.
The $GG$-equation gives, exactly as in Lemma \ref{lem:even-solve}, $g(m,i)=Am+Bi+C$ and $f'(m,i)=Am+Bi+2C$, hence $f(m,i)=Am+Bi+2C-\varepsilon$ for $m\neq0$.
The $LL$-equations at $(1,0,1,1)$ and $(1,0,2,0)$ and the $LG$-equation at $(1,0,1,0)$ give
$$\beta B=2\varepsilon-2C,\qquad q(2\varepsilon-2C)=\beta(2C-\varepsilon),\qquad q(\varepsilon-C)=\beta(C-\varepsilon).$$
Hence $\beta\varepsilon=0$. If $\beta=0$ then $\varepsilon=C$; if $\beta\neq0$ then $\varepsilon=0$, $C(\beta+q)=0$ and $\beta B=-2C$.
The $LL$-equation with $m=0$ extends the formula for $f$ to $m=0$, except for the coefficient $L_{0,-q-\beta}\mapsto L_{0,-q}$. Consequently, up to coefficients with target $L_{0,-q}$, the component is
\begin{itemize}
\item $\frac Aq\ad_{L_{0,0}}+B\,D_1+C\,\id$ if $\beta=0$;
\item $\frac A{\beta+q}\ad_{L_{0,\beta}}$ if $\beta\neq0,-q$;
\item $A\,D_3+B\,D_4$ if $\beta=-q$ (possible only if $q\in\mathbb Z^*$).
\end{itemize}

\emph{Odd case.} The $LG$-equation with $m=n=0$ gives $f(0,\cdot)=0$, so $u_\beta=0$, and $D'=D$ on all basis vectors with $m\neq 0$.
The remaining equations give, as in Lemma \ref{lem:odd-solve}, that $f(n,\cdot)$ and $g(n,\cdot)$ are constant for $n\neq0$, $g(n,0)=\Gamma$ for all $n\neq0$,
$2qf(m,i)=\Gamma m(\beta+\frac q2)$, and $g(0,j)=\Gamma$ unless $j+\beta+q=0$. Hence, up to the coefficient $G_{0,-q-\beta}\mapsto L_{0,-q}$, the component is $\frac{\Gamma}{2q}\ad_{G_{0,\beta}}$.

\medskip
\noindent\textbf{Step 5: assembly.}
Put $y=\sum_\beta\frac{A_\beta}{\beta+q}L_{0,\beta}$ (sum over $\beta\neq-q$) in the even case and $y=\sum_\beta\frac{\Gamma_\beta}{2q}G_{0,\beta}$ in the odd case.
These sums are finite, since $A_\beta$ and $\Gamma_\beta$ are coefficients of the vectors $D(G_{1,0})$ and $D(G_{0,0})$.
By Steps 3--4, all coefficients of
$$D-\ad_y-\big(B_0 D_1+C_0\,\id+A_{-q} D_3+B_{-q} D_4\big)\quad\text{(even case)},\qquad D-\ad_y\quad\text{(odd case)}$$
vanish, except those with target $L_{0,-q}$ and source in $\cS(q)_\theta$; here the constants are those of the components $\beta=0$ and $\beta=-q$,
and the terms $A_{-q} D_3+B_{-q} D_4$ occur only if $q\in\mathbb Z^*$.
If $q\notin\mathbb Z$, there are no such coefficients; if $q\in\mathbb Z^*$, the difference belongs to $\Ann^Q_\theta(\cS(q))$.

\medskip
\noindent\textbf{Step 6: the sums are direct.}
By Theorem \ref{thm:main} it is enough to show the following.
Suppose $\Delta+l\,\id+Z=0$ with $\Delta\in\Der_{\bar0}(\cS(q))$, $l\in\mathbb C$, $Z\in\Ann^Q_{\bar0}$. Applying the derivation identity for $\Delta$ to
$[G_{0,0},G_{m,i}]=2qL_{m,i}$ and using $Z(G_{n,j})=0$, we obtain $Z(L_{m,i})=l\,L_{m,i}$; at $L_{1,0}$ this gives $l=0$, hence $Z=0$ and $\Delta=0$.
If $\ad_x+Z=0$ with $x\in\cS(q)_{\bar1}$ and $Z\in\Ann^Q_{\bar1}$, then comparing the coefficients of $L_{1,i}$ in the images of $G_{1-m,0}$ gives
$2q\,x_{G_{m,i}}=0$, so $x=0$ and $Z=0$. \hfill$\Box$

\subsection{\normalsize Quasi-derivations of $\cS(0)$}\label{sec:qder-q0}

For $q=0$ we use the form $s$ from \eqref{eq:B0}: $[L_a,L_b]=s(a,b)L_{a+b}$, $[L_a,G_b]=s(a,b)G_{a+b}$ and $[G_a,G_b]=0$.
The center of $\cS(0)$ is $Z(\cS(0))=\langle L_{0,0},G_{0,0}\rangle$, and we set
\[
\Ann^Q(\cS(0))=\Ann^Q_{\bar0}\oplus\Ann^Q_{\bar1},\qquad
\Ann^Q_\theta=\big\{g\in\mathfrak{gl}_\theta(\cS(0)) : g(\cS(0))\subseteq\langle L_{0,0},G_{0,0}\rangle\big\}.
\]
Thus an element of $\Ann^Q_{\bar0}$ sends every $L_{m,i}$ to a multiple of $L_{0,0}$ and every $G_{m,i}$ to a multiple of $G_{0,0}$,
while an element of $\Ann^Q_{\bar1}$ sends every $L_{m,i}$ to a multiple of $G_{0,0}$ and every $G_{m,i}$ to a multiple of $L_{0,0}$;
these multiples can be chosen arbitrarily. In particular $E_L,E_G\in\Ann^Q_{\bar0}$ and $O_5,\Th_0\in\Ann^Q_{\bar1}$.

\begin{proposition}\label{prop:q0}
For every $\beta=(\beta_1,\beta_2)\in\mathbb Z^2$ and $\kappa,\lambda\in\mathbb C$ the following maps are quasi-derivations of $\cS(0)$:
\begin{enumerate}
\item[{\rm 1.}] the even map $V_{\beta,\kappa,\lambda}\colon X_{m,i}\mapsto(\kappa m+\lambda i)X_{m+\beta_1,i+\beta_2}$, with
$V'_{\beta,\kappa,\lambda}(X_{m,i})=\big(\kappa(m+\beta_1)+\lambda(i+\beta_2)\big)X_{m+\beta_1,i+\beta_2}$;
\item[{\rm 2.}] the odd map $V^{o}_{\beta,\kappa,\lambda}\colon L_{m,i}\mapsto(\kappa m+\lambda i)G_{m+\beta_1,i+\beta_2}$, $G_{m,i}\mapsto0$;
\item[{\rm 3.}] $\id$, $T$, $\Th$, the odd map $\Th'\colon L_{m,i}\mapsto G_{m,i}$, $G_{m,i}\mapsto 0$ (with $D'=2\Th'$),
and every element of $\Ann^Q(\cS(0))$.
\end{enumerate}
Moreover, $V_{\beta,\kappa,\lambda}$ is a derivation if and only if $\kappa\beta_1+\lambda\beta_2=0$.
\end{proposition}

\begin{proof}
For $\ell(m,i)=\kappa m+\lambda i$, statement (1) reduces to the identity
\begin{center}
$s(a,b)\ell(a+b+\beta)=\ell(a)s(a+\beta,b)+\ell(b)s(a,b+\beta)$,
\end{center} which is equivalent to
$\ell(\beta)s(a,b)=\ell(a)s(\beta,b)+\ell(b)s(a,\beta)$; the latter holds because an alternating trilinear form on a $2$-dimensional space vanishes.
The other statements are direct verifications. For the last assertion, compare coefficients in the derivation identity for $L_{1,0}$ and $L_{0,1}$.
\end{proof}

Since $\ad_{L_\beta}=V_{\beta,\beta_2,-\beta_1}$ on $\cS(0)$, every $\beta\neq0$ gives quasi-derivations $V_{\beta,\kappa,\lambda}$ which are not derivations.
Hence $\QDer(\cS(0))/\Der(\cS(0))$ is infinite-dimensional.
We now show that the maps of Proposition \ref{prop:q0} exhaust all quasi-derivations of $\cS(0)$.
Since \begin{center}
$V_{\beta,\kappa,\lambda}+V_{\beta,\kappa',\lambda'}=V_{\beta,\kappa+\kappa',\lambda+\lambda'}$
\end{center} and similarly for $V^o$, the spaces
\[
\mathcal V_{\bar0}=\big\langle V_{\beta,\kappa,\lambda} : \beta\in\mathbb Z^2,\ \kappa,\lambda\in\mathbb C\big\rangle,\qquad
\mathcal V_{\bar1}=\big\langle V^o_{\beta,\kappa,\lambda} : \beta\in\mathbb Z^2,\ \kappa,\lambda\in\mathbb C\big\rangle
\]
consist of the finite sums $\sum_{\beta\in S}V_{\beta,\kappa_\beta,\lambda_\beta}$ (respectively $\sum_{\beta\in S}V^o_{\beta,\kappa_\beta,\lambda_\beta}$)
over finite sets $S\subset\mathbb Z^2$.

\begin{theorem}\label{thm:q0}
\begin{align*}
\QDer_{\bar0}(\cS(0))&=\mathcal V_{\bar0}\oplus\langle\id,\ T\rangle\oplus\Ann^Q_{\bar0}(\cS(0)),\\
\QDer_{\bar1}(\cS(0))&=\mathcal V_{\bar1}\oplus\langle\Th,\ \Th'\rangle\oplus\Ann^Q_{\bar1}(\cS(0)).
\end{align*}
In other words, a linear map $D\in\mathfrak{gl}_{\bar0}(\cS(0))$ is a quasi-derivation if and only if
\begin{center}$D=\sum_{\beta\in S}V_{\beta,\kappa_\beta,\lambda_\beta}+a\,\id+t\,T+Z$
\end{center} for a finite set $S\subset\mathbb Z^2$, constants $\kappa_\beta,\lambda_\beta,a,t\in\mathbb C$
and $Z\in\Ann^Q_{\bar0}(\cS(0))$; and $D\in\mathfrak{gl}_{\bar1}(\cS(0))$ is a quasi-derivation if and only if
$D=\sum_{\beta\in S}V^o_{\beta,\kappa_\beta,\lambda_\beta}+b\,\Th+c\,\Th'+Z$ with $Z\in\Ann^Q_{\bar1}(\cS(0))$.
These decompositions are unique.
\end{theorem}

For the proof we need two functional equations. For functions $f,F\colon\mathbb Z^2\to\mathbb C$, $k,K\colon\mathbb Z^2\to\mathbb C$ and $\sigma\in\mathbb Z^2$ consider
\begin{align}
s(\alpha,\beta)F(\alpha+\beta)&=f(\alpha)\,s(\alpha+\sigma,\beta)+s(\alpha,\beta+\sigma)f(\beta),\tag{Q$_\sigma$}\label{eq:Q}\\
s(\alpha,\beta)K(\alpha+\beta)&=s(\alpha,\beta+\sigma)k(\beta)\tag{C$_\sigma$}\label{eq:C}
\end{align}
for all $\alpha,\beta\in\mathbb Z^2$. They differ from (E$_\sigma$) and (Z$_\sigma$) in that the function on the left-hand side is not required to coincide with
the function on the right-hand side; this reflects the fact that $D'$ need not be equal to $D$.
We shall repeatedly use the following elementary fact: if $x_1,\dots,x_r\in\mathbb Z^2$ are nonzero and $c_1,\dots,c_r\in\mathbb C$, then there is
$\gamma\in\mathbb Z^2$ with $s(x_j,\gamma)\neq c_j$ for all $j$, because each condition $s(x_j,\gamma)=c_j$ defines a line, and $\mathbb Z^2$ is not a
union of finitely many lines. Recall also that if $s(\sigma,\rho)\neq0$ and $s(\sigma,\gamma)=s(\rho,\gamma)=0$, then $\gamma=0$.

\begin{lemma}\label{lem:qcent}
Let $(k,K)$ satisfy \textup{(C$_\sigma$)}.
\begin{enumerate}
\item[{\rm (1)}] If $\sigma\neq0$, then $k(\beta)=0$ for all $\beta\neq-\sigma$.
\item[{\rm (2)}] If $\sigma=0$, then $k$ is constant on $\mathbb Z^2\setminus\{0\}$.
\end{enumerate}
\end{lemma}

\begin{proof}
(1) Equation (C$_\sigma$) at $(-\beta,\beta)$ gives $s(\sigma,\beta)k(\beta)=0$, so $k(\beta)=0$ whenever $s(\sigma,\beta)\neq0$.
Let $\gamma\neq0$ and choose $\beta$ with $s(\sigma,\beta)\neq0$ and $s(\gamma,\beta)\neq0$. Then (C$_\sigma$) at $(\gamma-\beta,\beta)$ gives\begin{center}
$s(\gamma,\beta)K(\gamma)=s(\gamma-\beta,\beta+\sigma)k(\beta)=0$,
\end{center} so $K(\gamma)=0$ for all $\gamma\neq0$.
Finally, let $s(\sigma,\beta)=0$ and $\beta\neq-\sigma$. Choose $\alpha$ with $s(\alpha,\beta+\sigma)\neq0$. Then $\alpha+\beta\neq0$, since
$s(-\beta,\beta+\sigma)=-s(\beta,\sigma)=0$. Now (C$_\sigma$) at $(\alpha,\beta)$ gives $0=s(\alpha,\beta+\sigma)k(\beta)$, so $k(\beta)=0$.

(2) Equation (C$_0$) at $(\gamma-\beta,\beta)$ gives $s(\gamma,\beta)K(\gamma)=s(\gamma,\beta)k(\beta)$, so $k(\beta)=K(\gamma)$ whenever $s(\gamma,\beta)\neq0$.
For nonzero $\beta,\beta'$ choose $\gamma$ with $s(\gamma,\beta)\neq0\neq s(\gamma,\beta')$; then $k(\beta)=K(\gamma)=k(\beta')$.
\end{proof}

\begin{lemma}\label{lem:qder-ne}
Let $\sigma\neq0$ and let $(f,F)$ satisfy \textup{(Q$_\sigma$)}. Then there are $\kappa,\lambda\in\mathbb C$ such that
$f(\alpha)=\kappa\alpha_1+\lambda\alpha_2$ for all $\alpha\neq-\sigma$.
\end{lemma}

\begin{proof}
Choose $\rho$ with $s(\sigma,\rho)\neq0$ and let $\ell(\alpha)=\big(f(\sigma)s(\alpha,\rho)-f(\rho)s(\alpha,\sigma)\big)/s(\sigma,\rho)$.
Then $\ell$ is additive, $\ell(\sigma)=f(\sigma)$ and $\ell(\rho)=f(\rho)$, and, by the identity used in the proof of Proposition \ref{prop:q0},
the pair $(\ell,\ell(\cdot+\sigma))$ satisfies (Q$_\sigma$). Hence $g=f-\ell$ and $G=F-\ell(\cdot+\sigma)$ satisfy (Q$_\sigma$), and $g(\sigma)=g(\rho)=0$.
It suffices to prove that $g(\alpha)=0$ for $\alpha\neq-\sigma$. Put $N=\{\alpha : s(\sigma,\alpha)\neq0\}$; note that $N+\mathbb Z\sigma=N$.
\begin{enumerate}
\item (Q$_\sigma$) at $(\gamma+\sigma,-\sigma)$ gives $s(\gamma,\sigma)G(\gamma)=s(\gamma,\sigma)g(\gamma+\sigma)$, so $G(\gamma)=g(\gamma+\sigma)$ for $\gamma\in N$.
\item (Q$_\sigma$) at $(\alpha,\sigma)$, $\alpha\in N$, together with (1) and $g(\sigma)=0$, gives $g(\alpha+2\sigma)=g(\alpha)$.
\item (Q$_\sigma$) at $(\alpha,-2\sigma)$, $\alpha\in N$, together with (1), gives $g(\alpha-\sigma)=g(\alpha)+\frac12g(-2\sigma)$.
\item Applying (3) to $\alpha+\sigma$ and to $\alpha+2\sigma$ and using (2), we get $g(\alpha)=g(\alpha+\sigma)+\frac12g(-2\sigma)$ and
$g(\alpha+\sigma)=g(\alpha)+\frac12g(-2\sigma)$. Hence $g(-2\sigma)=0$ and $g(\alpha+\sigma)=g(\alpha)$ for all $\alpha\in N$.
\item Let $\rho'\in N$ with $g(\rho')=0$, and let $\alpha,\alpha+\rho'\in N$. By (1) and (4), $G(\alpha+\rho')=G(\alpha+\rho'+\sigma)=g(\alpha+\rho')$ and $g(\rho'+\sigma)=0$.
Then (Q$_\sigma$) at $(\alpha,\rho')$ and at $(\alpha,\rho'+\sigma)$ read
\begin{align*}
s(\alpha,\rho')\,g(\alpha+\rho')&=g(\alpha)\big(s(\alpha,\rho')+s(\sigma,\rho')\big),\\
\big(s(\alpha,\rho')+s(\alpha,\sigma)\big)g(\alpha+\rho')&=g(\alpha)\big(s(\alpha,\rho')+s(\alpha,\sigma)+s(\sigma,\rho')\big).
\end{align*}
Subtracting, $s(\alpha,\sigma)\big(g(\alpha+\rho')-g(\alpha)\big)=0$, so $g(\alpha+\rho')=g(\alpha)$, and then the first equation gives $s(\sigma,\rho')g(\alpha)=0$, i.e.\ $g(\alpha)=0$.
\item Applying (5) with $\rho'=\rho$ and $\alpha=2\rho$ (note that $2\rho,3\rho\in N$) gives $g(2\rho)=0$. Let $\alpha\in N$. If $\alpha+\rho\in N$, then (5) with $\rho'=\rho$ gives $g(\alpha)=0$;
otherwise $s(\sigma,\alpha+2\rho)=s(\sigma,\rho)\neq0$, and (5) with $\rho'=2\rho$ gives $g(\alpha)=0$. Thus $g=0$ on $N$.
\item Let $s(\sigma,\gamma)=0$ and $\gamma\neq-\sigma$. Then $\rho+\gamma\in N$, so $G(\rho+\gamma)=g(\rho+\gamma+\sigma)=0$ by (1) and (6), and
(Q$_\sigma$) at $(\rho,\gamma)$ gives $s(\rho,\gamma+\sigma)g(\gamma)=0$. Since $s(\sigma,\gamma+\sigma)=0$ and $\gamma+\sigma\neq0$, we have $s(\rho,\gamma+\sigma)\neq0$, so $g(\gamma)=0$.
\end{enumerate}
Hence $f=\ell$ outside $-\sigma$, and $\ell(\alpha)=\kappa\alpha_1+\lambda\alpha_2$ with $\kappa=\ell(1,0)$, $\lambda=\ell(0,1)$.
\end{proof}

\begin{lemma}\label{lem:qder-zero}
Let $(f,F)$ satisfy \textup{(Q$_0$)}. Then there are $\kappa,\lambda,c\in\mathbb C$ such that
$f(\alpha)=\kappa\alpha_1+\lambda\alpha_2+c$ for all $\alpha\neq0$.
\end{lemma}

\begin{proof}
Equation (Q$_0$) says that $F(\alpha+\beta)=f(\alpha)+f(\beta)$ whenever $s(\alpha,\beta)\neq0$.
We claim that $F-f$ is constant on $\mathbb Z^2\setminus\{0\}$. Let $x,y\neq0$ and choose $\beta$ with
$s(x,\beta)\notin\{0,s(x,y)\}$ and $s(y,\beta)\notin\{0,s(y,x)\}$. Then the four pairs $(x,y-\beta)$, $(x-\beta,\beta)$, $(x-\beta,y)$, $(\beta,y-\beta)$
have nonzero value of $s$, so
$$
f(x)+f(y-\beta)=F(x+y-\beta)=f(x-\beta)+f(y),$$$$ F(x)=f(x-\beta)+f(\beta),\quad F(y)=f(\beta)+f(y-\beta).
$$
Hence $F(x)-F(y)=f(x-\beta)-f(y-\beta)=f(x)-f(y)$, which proves the claim. Let $c$ be the constant value of $F-f$ on $\mathbb Z^2\setminus\{0\}$
and put $h=f-c$. If $s(\alpha,\beta)\neq0$, then $\alpha+\beta\neq0$, so $h(\alpha+\beta)=F(\alpha+\beta)-2c=h(\alpha)+h(\beta)$.
Thus $h$ satisfies (E$_0$), and Lemma \ref{lem:der-zero} gives $h(\alpha)=\kappa\alpha_1+\lambda\alpha_2$ for $\alpha\neq0$, with $\kappa=h(1,0)$, $\lambda=h(0,1)$.
\end{proof}

\begin{proof}[Proof of Theorem \ref{thm:q0}]
All maps on the right-hand sides are quasi-derivations by Proposition \ref{prop:q0}, and $\QDer_\theta(\cS(0))$ is a vector space.
Conversely, let $D,D'\in\mathfrak{gl}_\theta(\cS(0))$ satisfy \eqref{qder}.

\emph{Even case.} For $\sigma\in\mathbb Z^2$ put $f_\sigma(\alpha)=D(L_\alpha)_{L_{\alpha+\sigma}}$, $g_\sigma(\alpha)=D(G_\alpha)_{G_{\alpha+\sigma}}$,
$F_\sigma(\alpha)=D'(L_\alpha)_{L_{\alpha+\sigma}}$ and $H_\sigma(\alpha)=D'(G_\alpha)_{G_{\alpha+\sigma}}$.
Comparing the coefficients of $L_{\alpha+\beta+\sigma}$ in \eqref{qder} for the pair $(L_\alpha,L_\beta)$ shows that $(f_\sigma,F_\sigma)$ satisfies (Q$_\sigma$).
The pair $(L_\alpha,G_\beta)$ gives
\[
s(\alpha,\beta)H_\sigma(\alpha+\beta)=f_\sigma(\alpha)s(\alpha+\sigma,\beta)+s(\alpha,\beta+\sigma)g_\sigma(\beta),
\]
and subtracting (Q$_\sigma$) we see that $(g_\sigma-f_\sigma,H_\sigma-F_\sigma)$ satisfies (C$_\sigma$).
By Lemmas \ref{lem:qcent}, \ref{lem:qder-ne} and \ref{lem:qder-zero} there are constants $\kappa_\sigma,\lambda_\sigma$ such that, with
$\ell_\sigma(\alpha)=\kappa_\sigma\alpha_1+\lambda_\sigma\alpha_2$,
\[
f_\sigma(\alpha)=g_\sigma(\alpha)=\ell_\sigma(\alpha)\quad(\sigma\neq0,\ \alpha\neq-\sigma),\]
\[
f_0(\alpha)=\ell_0(\alpha)+a,\quad g_0(\alpha)=\ell_0(\alpha)+a+t\quad(\alpha\neq0)
\]
for some $a,t\in\mathbb C$. For $\sigma\notin\{0,(-1,0),(0,-1)\}$ we have $\kappa_\sigma=f_\sigma(1,0)$ and $\lambda_\sigma=f_\sigma(0,1)$, which are the
coefficients of $L_{(1,0)+\sigma}$ in $D(L_{1,0})$ and of $L_{(0,1)+\sigma}$ in $D(L_{0,1})$. Hence the set $S$ of those $\sigma$ with $(\kappa_\sigma,\lambda_\sigma)\neq(0,0)$ is finite.
Put
\[
Z=D-\sum_{\sigma\in S}V_{\sigma,\kappa_\sigma,\lambda_\sigma}-a\,\id-t\,T.
\]
By the formulas above, the coefficient of $X_{\alpha+\sigma}$ in $Z(X_\alpha)$ ($X\in\{L,G\}$) vanishes whenever $\alpha+\sigma\neq0$.
Thus $Z(L_\alpha)\in\langle L_{0,0}\rangle$ and $Z(G_\alpha)\in\langle G_{0,0}\rangle$ for all $\alpha$, i.e.\ $Z\in\Ann^Q_{\bar0}(\cS(0))$.

\emph{Odd case.} Put $u_\sigma(\alpha)=D(L_\alpha)_{G_{\alpha+\sigma}}$, $v_\sigma(\alpha)=D(G_\alpha)_{L_{\alpha+\sigma}}$,
$U_\sigma(\alpha)=D'(L_\alpha)_{G_{\alpha+\sigma}}$ and $W_\sigma(\alpha)=D'(G_\alpha)_{L_{\alpha+\sigma}}$.
Since $[G_{\alpha+\sigma},L_\beta]=s(\alpha+\sigma,\beta)G_{\alpha+\beta+\sigma}$, the pair $(L_\alpha,L_\beta)$ shows that $(u_\sigma,U_\sigma)$ satisfies (Q$_\sigma$).
Since $[G,G]=0$, the pair $(L_\alpha,G_\beta)$ gives $s(\alpha,\beta)W_\sigma(\alpha+\beta)=s(\alpha,\beta+\sigma)v_\sigma(\beta)$, i.e.\ $(v_\sigma,W_\sigma)$ satisfies (C$_\sigma$).
By Lemmas \ref{lem:qcent}, \ref{lem:qder-ne} and \ref{lem:qder-zero},
\[
u_\sigma(\alpha)=\ell_\sigma(\alpha),\quad v_\sigma(\alpha)=0\quad(\sigma\neq0,\ \alpha\neq-\sigma),\qquad
u_0(\alpha)=\ell_0(\alpha)+c,\quad v_0(\alpha)=b\quad(\alpha\neq0),
\]
with $\ell_\sigma(\alpha)=\kappa_\sigma\alpha_1+\lambda_\sigma\alpha_2$. As before, only finitely many $(\kappa_\sigma,\lambda_\sigma)$ are nonzero, and
$Z=D-\sum_{\sigma\in S}V^o_{\sigma,\kappa_\sigma,\lambda_\sigma}-b\,\Th-c\,\Th'$ satisfies $Z(L_\alpha)\in\langle G_{0,0}\rangle$ and $Z(G_\alpha)\in\langle L_{0,0}\rangle$,
i.e.\ $Z\in\Ann^Q_{\bar1}(\cS(0))$.

\emph{Uniqueness.} Suppose that $\sum_{\sigma\in S}V_{\sigma,\kappa_\sigma,\lambda_\sigma}+a\,\id+t\,T+Z=0$ with $Z\in\Ann^Q_{\bar0}$.
For fixed $\sigma$, the coefficient of $L_{\alpha+\sigma}$ in the image of $L_\alpha$ gives $\kappa_\sigma\alpha_1+\lambda_\sigma\alpha_2+a\,\delta_{\sigma,0}=0$
for all $\alpha\neq-\sigma$. Since $\mathbb Z^2\setminus\{-\sigma\}$ contains three affinely independent points, $\kappa_\sigma=\lambda_\sigma=0$ for all $\sigma$ and $a=0$.
The coefficient of $G_{\alpha}$ in the image of $G_\alpha$, $\alpha\neq0$, then gives $t=0$, and finally $Z=0$.
The odd case is identical: the coefficients of $G_{\alpha+\sigma}$ in the images of $L_\alpha$ give $\kappa_\sigma=\lambda_\sigma=c=0$,
the coefficient of $L_\alpha$ in the image of $G_\alpha$, $\alpha\neq0$, gives $b=0$, and then $Z=0$.
\end{proof}

\begin{remark}\label{rem:q0-der}
On $\cS(0)$ we have $\ad_{L_\beta}=V_{\beta,\beta_2,-\beta_1}$, $\ad_{G_\beta}=V^o_{\beta,\beta_2,-\beta_1}$, $D_1=V_{0,0,1}$, $D_2=V_{0,1,0}$,
$O_1=V^o_{0,0,1}$, $O_2=V^o_{0,1,0}$, while $E_L,E_G\in\Ann^Q_{\bar0}$ and $O_5,\Th_0\in\Ann^Q_{\bar1}$.
For $\beta\neq0$ the vectors $(\beta_2,-\beta_1)$ and $(\beta_1,\beta_2)$ are linearly independent, so every $V_{\beta,\kappa,\lambda}$ is a linear combination of
$\ad_{L_\beta}$ and $W_\beta=V_{\beta,\beta_1,\beta_2}$, and similarly for $V^o_{\beta,\kappa,\lambda}$, $\ad_{G_\beta}$ and $W^o_\beta=V^o_{\beta,\beta_1,\beta_2}$.
Hence, by Theorems \ref{thm:main} and \ref{thm:q0},
\begin{align*}
\QDer_{\bar0}(\cS(0))&=\Der_{\bar0}(\cS(0))+\big\langle W_\beta : \beta\in\mathbb Z^2\setminus\{0\}\big\rangle+\langle\id\rangle+\Ann^Q_{\bar0}(\cS(0)),\\
\QDer_{\bar1}(\cS(0))&=\Der_{\bar1}(\cS(0))+\big\langle W^o_\beta : \beta\in\mathbb Z^2\setminus\{0\}\big\rangle+\langle\Th'\rangle+\Ann^Q_{\bar1}(\cS(0)).
\end{align*}
In particular, for $q=0$ (in contrast with the case $q\neq0$) not every quasi-derivation is the sum of a derivation, a multiple of $\id$ and a map with values in the center.
\end{remark}

\subsection{\normalsize Corollaries}

Summarizing Theorems \ref{thm:main}, \ref{thm:qnotZ} and \ref{thm:qZ}, we have the following statement.

\begin{theoremM}
Let $q\in\mathbb C^*$. Then
$$\QDer\big(\cS(q)\big)=\Inn\big(\cS(q)\big)\oplus\langle\id\rangle\oplus\langle D_1\rangle\oplus\langle D_3,D_4\rangle\oplus\Ann^Q\big(\cS(q)\big),$$
where $\langle D_3,D_4\rangle$ and $\Ann^Q(\cS(q))$ occur only for $q\in\mathbb Z^*$.
In particular, $\QDer(\cS(q))=\Der(\cS(q))\oplus\langle\id\rangle$ for $q\notin\mathbb Z$, and
$\QDer(\cS(q))=\Der(\cS(q))+\langle\id\rangle+\Ann^Q(\cS(q))$ for $q\in\mathbb Z^*$.
\end{theoremM}

\begin{definition}
A nonzero $\delta$-derivation $\varphi$ is \textit{nontrivial} if $\delta\notin\{0;1\}$ and $\varphi\notin\langle\id\rangle$.
\end{definition}

\begin{theorem}\label{thm:delta}
Let $q\neq0$ and $\delta\notin\{0;1\}$. Then
$$\Der_{\delta,\bar0}(\cS(q))=\begin{cases}\langle\id\rangle, & \delta=\frac12,\\ 0, & \delta\neq\frac12,\end{cases}\qquad
\Der_{\delta,\bar1}(\cS(q))=\begin{cases}\langle D_6\rangle, & q\in2\mathbb Z^*,\\ 0, & q\notin2\mathbb Z.\end{cases}$$
In particular, $\cS(q)$ admits nontrivial homogeneous $\delta$-derivations if and only if $q\in2\mathbb Z^*$, and they are the nonzero multiples of $D_6$
(for $\delta=\frac12$, the nontrivial $\frac12$-derivations are the maps $l\,\id+c\,D_6$ with $c\neq0$);
the map $D_6$ is a $\delta$-derivation for every $\delta$.
\end{theorem}

\begin{proof}
\emph{Even part.} Let $\varphi\in\Der_{\delta,\bar0}(\cS(q))$. Then $\varphi$ is a quasi-derivation with $\varphi'=\delta^{-1}\varphi$, so by
Theorems \ref{thm:qnotZ} and \ref{thm:qZ}, $\varphi=\Delta+l\,\id+Z$ with $\Delta\in\Der_{\bar0}(\cS(q))$, $l\in\mathbb C$ and $Z\in\Ann^Q_{\bar0}(\cS(q))$.
For basis vectors $X_a,X_b$, the $\delta$-derivation identity together with the derivation identity for $\Delta$ and the centrality of the values of $Z$ gives
$$c(X_a,X_b)\Big((1-\delta)\Delta(X'_{a+b})+(1-2\delta)l\,X'_{a+b}+Z(X'_{a+b})\Big)=0.$$
Every basis vector $X_c\neq G_{0,-\frac{3q}{2}}$ is a nonzero multiple of a multiplication of two basis vectors: by \eqref{LG} and \eqref{GG},
\[
L_{m,i}=\tfrac{1}{2q}[G_{0,0},G_{m,i}],\quad G_{m,i}=\tfrac{1}{mq}[L_{0,0},G_{m,i}]\ (m\neq0),\quad
G_{0,i}=-\tfrac{1}{i+\frac{3q}{2}}[L_{1,0},G_{-1,i}]\ \bigl(i\neq-\tfrac{3q}2\bigr).
\]
Hence
\begin{equation}\label{psi}
(1-\delta)\Delta(X_c)+(1-2\delta)l\,X_c+Z(X_c)=0\quad\text{for every basis vector } X_c\neq G_{0,-\frac{3q}{2}}.
\end{equation}
(i) Comparing the coefficients of $L_{1,0}$ in the derivation identity for $[L_{0,0},L_{1,0}]=qL_{1,0}$ gives $\Delta(L_{0,0})_{L_{0,0}}=0$.
Since $Z(L_{0,0})\in\langle L_{0,-q}\rangle$ and $q\neq0$, the coefficient of $L_{0,0}$ in \eqref{psi} for $X_c=L_{0,0}$ gives $(1-2\delta)l=0$.
(ii) For odd $X_c\neq G_{0,-\frac{3q}2}$ we have $Z(X_c)=0$, hence $\Delta(X_c)=0$ by \eqref{psi}, as $\delta\neq1$.
(iii) Every $L_{m,i}$ is a nonzero multiple of $[G_{0,0},G_{m,i}]$, or, if $(m,i)=(0,-\frac{3q}2)$, of $[G_{1,0},G_{-1,i}]$; in both cases $\Delta$ kills both factors
by (ii), so $\Delta(L_{m,i})=0$. Then \eqref{psi} gives $Z(L_{m,i})=0$, i.e. $Z=0$.
(iv) Finally, if $j_0=-\frac{3q}2\in\mathbb Z$, applying $\Delta$ to $[G_{0,0},G_{0,j_0}]=2qL_{0,j_0}$ gives $[G_{0,0},\Delta(G_{0,j_0})]=0$;
since $\ad_{G_{0,0}}\colon G_{n,j}\mapsto 2qL_{n,j}$ is injective on $\cS(q)_{\bar1}$, $\Delta(G_{0,j_0})=0$.
Thus $\Delta=0$ and $\varphi=l\,\id$ with $(1-2\delta)l=0$. Conversely, $l\,\id$ is a $\delta$-derivation if and only if $l=2\delta l$.

\emph{Odd part.} Let $\varphi\in\Der_{\delta,\bar1}(\cS(q))$. As above, $\varphi=\ad_x+Z$ with $x\in\cS(q)_{\bar1}$ and $Z\in\Ann^Q_{\bar1}(\cS(q))$, and
$(1-\delta)[x,X_c]+Z(X_c)=0$ for every basis vector $X_c\neq G_{0,-\frac{3q}2}$.
Since $Z$ vanishes on $\cS(q)_{\bar0}$, we get $[x,L_{m,i}]=0$ for all $m,i$. Writing $x=\sum x_{n,j}G_{n,j}$ and using
$[G_{n,j},L_{0,0}]=-nq\,G_{n,j}$ we obtain $x_{n,j}=0$ for $n\neq0$. Next, $(1-\delta)[x,G_{1,0}]=-Z(G_{1,0})\in\langle L_{0,-q}\rangle$, while
$[x,G_{1,0}]=2q\sum_j x_{0,j}L_{1,j}$; hence $x=0$. Therefore $\varphi=Z$ and $Z(X_c)=0$ for all $X_c\neq G_{0,-\frac{3q}2}$.
If $q\notin\mathbb Z$, then $Z=0$. If $q\in\mathbb Z^*$, then $Z$ is determined by $Z(G_{0,-\frac{3q}2})\in\langle L_{0,-q}\rangle$; this vector exists only if
$\frac{3q}{2}\in\mathbb Z$, i.e. $q\in2\mathbb Z^*$, and then $Z\in\langle D_6\rangle$.
Conversely, for $q\in2\mathbb Z^*$ the map $D_6$ vanishes on $[\cS(q),\cS(q)]$ (as shown in the proof of Lemma \ref{lem:outer-are-der}, $G_{0,-\frac{3q}2}$ does not occur
in any multiplication of basis vectors with a nonzero coefficient) and takes values in the center, so both sides of the
$\delta$-derivation identity vanish.
\end{proof}

\begin{corollary}
Let $q\neq0$. Then $\Der_{\frac12}(\cS(q))=\langle\id\rangle$ if $q\notin2\mathbb Z$, and $\Der_{\frac12}(\cS(q))=\langle\id,D_6\rangle$ if $q\in2\mathbb Z^*$.
For $\delta\notin\{0;\frac12;1\}$, $\Der_\delta(\cS(q))=\langle D_6\rangle$ if $q\in 2\mathbb Z^*$, and $\Der_\delta(\cS(q))=0$ otherwise.
\end{corollary}

\begin{theorem}\label{thm:tdp}
Let $q\neq0$ and $\delta\notin\{0;1\}$. Then every graded supercommutative bilinear product $\cdot$ on $\cS(q)$ satisfying \eqref{tdp} is zero.
In particular, $\cS(q)$ does not admit nontrivial transposed $\delta$-Poisson superalgebra structures.
\end{theorem}

\begin{proof}
By Lemma \ref{lem:Lz} and Theorem \ref{thm:delta}, for every even basis vector $X_c$ there is $l_c$ with $\mathrm L_{X_c}=l_c\,\id$,
and for every odd basis vector $X_c$ the map $\mathrm L_{X_c}$ vanishes on all basis vectors except possibly $G_{0,-\frac{3q}2}$
(it is $0$ if $q\notin2\mathbb Z$ and a multiple of $D_6$ if $q\in2\mathbb Z^*$).
Let $X_c$ be even. Then $l_c\,G_{1,1}=X_c\cdot G_{1,1}=G_{1,1}\cdot X_c=\mathrm L_{G_{1,1}}(X_c)=0$, so $\mathrm L_{X_c}=0$.
Let $X_c$ be odd and $X_d$ a basis vector. If $X_d\neq G_{0,-\frac{3q}{2}}$ then $X_c\cdot X_d=0$. If $X_d=G_{0,-\frac{3q}{2}}$ and $X_c\neq X_d$,
then $X_c\cdot X_d=-X_d\cdot X_c=0$. Finally $X_d\cdot X_d=-X_d\cdot X_d$ gives $X_d\cdot X_d=0$. Hence the product is zero.
\end{proof}

\begin{remark}
Theorem \ref{thm:tdp} does not use the associativity of the product.
It should be compared with the Block Lie algebra $\cB(q)$, $q\in\mathbb Z$, which admits (up to isomorphism) exactly one nontrivial transposed $\delta$-Poisson structure
$L_{0,-2q}\cdot L_{0,-2q}=L_{0,-q}$. In $\cS(q)$ the vector $L_{0,-2q}$ lies in the derived superalgebra, so this structure does not extend.
The new odd $\delta$-derivation $D_6$ does not produce a structure either, because of supercommutativity: 
\begin{center}
$G_{0,-\frac{3q}{2}}\cdot G_{0,-\frac{3q}{2}}=0$.
\end{center}
\end{remark}

\end{document}